\pdfoutput=1
\documentclass[10pt,oneside]{amsart}
\usepackage[
paper=a4paper,
headsep=20pt,text={136mm,208mm},centering,includehead
]{geometry}
\usepackage{graphicx}
\usepackage[dvipsnames,usenames]{xcolor}
\usepackage[colorlinks=true,urlcolor=ForestGreen,linkcolor=ForestGreen,citecolor=ForestGreen]{hyperref}
\hypersetup{
	linkbordercolor={1 0 0}, % set to white
	citebordercolor={0 1 0} % set to white
}
\usepackage{mathrsfs}
\usepackage{bbm}
\usepackage{mathtools}
\usepackage[UKenglish]{babel}
\usepackage{amssymb,amsxtra,dsfont}
\usepackage[noabbrev]{cleveref}
\usepackage{enumerate}
\usepackage{bbm}
\usepackage{relsize}
\usepackage{tikz}
\usetikzlibrary{trees,decorations.pathmorphing,decorations.markings,matrix,shapes}

\iftrue
\makeatletter
\def\@settitle{%
	\vspace*{-20pt}
	\begin{flushleft}%
		\baselineskip14\p@\relax
		\normalfont\bfseries\LARGE
		\@title
	\end{flushleft}%
}
\def\@setauthors{%
	\begingroup
	\def\thanks{\protect\thanks@warning}%
	\trivlist
	\large \@topsep30\p@\relax
	\advance\@topsep by -\baselineskip
	\item\relax
	\author@andify\authors
	\def\\{\protect\linebreak}%
	\authors
	\ifx\@empty\contribs
	\else
	,\penalty-3 \space \@setcontribs
	\@closetoccontribs
	\fi
	\normalfont
	\@setaddresses
	\endtrivlist
	\endgroup
}
\def\@setaddresses{\par
	\nobreak \begingroup\raggedright
	\small
	\def\author##1{\nobreak\addvspace\smallskipamount}%
	\def\\{\unskip, \ignorespaces}%
	\interlinepenalty\@M
	\def\address##1##2{\begingroup
		\par\addvspace\bigskipamount\noindent
		\@ifnotempty{##1}{(\ignorespaces##1\unskip) }%
		{\ignorespaces##2}\par\endgroup}%
	\def\curraddr##1##2{\begingroup
		\@ifnotempty{##2}{\nobreak\noindent\curraddrname
			\@ifnotempty{##1}{, \ignorespaces##1\unskip}\/:\space
			##2\par}\endgroup}%
	\def\email##1##2{\begingroup
		\@ifnotempty{##2}{\smallskip\nobreak\noindent E-mail address%
			\@ifnotempty{##1}{, \ignorespaces##1\unskip}\/:\space
			\ttfamily##2\par}\endgroup}%
	\def\urladdr##1##2{\begingroup
		\def~{\char`\~}%
		\@ifnotempty{##2}{\nobreak\noindent\urladdrname
			\@ifnotempty{##1}{, \ignorespaces##1\unskip}\/:\space
			\ttfamily##2\par}\endgroup}%
	\addresses
	\endgroup
	\global\let\addresses=\@empty
}
\def\@setabstracta{%
	\ifvoid\abstractbox
	\else
	\skip@25\p@ \advance\skip@-\lastskip
	\advance\skip@-\baselineskip \vskip\skip@
	\box\abstractbox
	\prevdepth\z@ % because \abstractbox is a vtop
	\vskip-15pt
	\fi
}
\renewenvironment{abstract}{%
	\ifx\maketitle\relax
	\ClassWarning{\@classname}{Abstract should precede
		\protect\maketitle\space in AMS document classes; reported}%
	\fi
	\global\setbox\abstractbox=\vtop \bgroup
	\normalfont\small
	\list{}{\labelwidth\z@
		\leftmargin0pc \rightmargin\leftmargin
		\listparindent\normalparindent \itemindent\z@
		\parsep\z@ \@plus\p@
		
	}%
	\item[\hskip\labelsep\bfseries\abstractname.]%
}{%
	\endlist\egroup
	\ifx\@setabstract\relax \@setabstracta \fi
}

\def\ps@headings{\ps@empty
	\def\@evenhead{%
		\setTrue{runhead}%
		\normalfont\scriptsize
		\rlap{\thepage}\hfill
		\def\thanks{\protect\thanks@warning}%
		\leftmark{}{}}%
	\def\@oddhead{%
		\setTrue{runhead}%
		\normalfont\scriptsize
		\def\thanks{\protect\thanks@warning}%
		\rightmark{}{}\hfill \llap{\thepage}}%
	\let\@mkboth\markboth
}\ps@headings

\def\section{\@startsection{section}{1}%
	\z@{-1.2\linespacing\@plus-.5\linespacing}{.8\linespacing}%
	{\normalfont\bfseries\Large}}
\def\subsection{\@startsection{subsection}{2}%
	\z@{-.8\linespacing\@plus-.3\linespacing}{.3\linespacing\@plus.2\linespacing}%
	{\normalfont\bfseries\large}}
\def\subsubsection{\@startsection{subsubsection}{3}%
	\z@{.7\linespacing\@plus.1\linespacing}{-1.5ex}%
	{\normalfont\bfseries}}
\def\@secnumfont{\bfseries}
\makeatother
\fi %\iftrue/false for amsart.cls hack

\makeatother

\newtheorem{theorem}{Theorem}[section]
\newtheorem*{theorem*}{Theorem}

\newtheorem{proposition}[theorem]{Proposition}
\newtheorem{corollary}[theorem]{Corollary}
\newtheorem{lemma}[theorem]{Lemma}

\theoremstyle{definition}
\newtheorem{definition}[theorem]{Definition}

\newtheorem*{remark}{Remark}

\numberwithin{equation}{section}

\newcommand{\lartial}{\mathbin{\rotatebox[origin=c]{180}{\(\partial\)}}}
\newcommand{\cdkh}{CDKh}
\newcommand{\dkh}{DKh}
\newcommand{\Q}{\mathbb{Q}}
\newcommand{\Z}{\mathbb{Z}}
\newcommand{\sg}{\mathfrak{s}}
\newcommand{\vup}{v^{\text{u}}_+}
\newcommand{\vum}{v^{\text{u}}_-}
\newcommand{\vlp}{v^{\text{l}}_+}
\newcommand{\vlm}{v^{\text{l}}_-}
\newcommand{\vulp}{v^{\text{u/l}}_+}
\newcommand{\vulm}{v^{\text{u/l}}_-}

\DeclareRobustCommand{\CloseDef}{%
	\leavevmode\unskip\penalty9999 \hbox{}\nobreak\hfill
	\quad\hbox{$\lozenge$}%
}

\begin{document}
	
\vspace*{-50pt}

\title[Doubled Khovanov Homology]{Doubled Khovanov Homology}

\author{William Rushworth}
\address{Department of Mathematical Sciences, Durham University, United Kingdom}
\email{william.rushworth@durham.ac.uk}

\def\subjclassname{\textup{2010} Mathematics Subject Classification}
\expandafter\let\csname subjclassname@1991\endcsname=\subjclassname
\expandafter\let\csname subjclassname@2000\endcsname=\subjclassname


\begin{abstract}
We define a homology theory of virtual links built out of the direct sum of the standard Khovanov complex with itself, motivating the name \textit{doubled Khovanov homology}. We demonstrate that it can be used to show that some virtual links are non-classical, and that it yields a condition on a virtual knot being the connect sum of two unknots. Further, we show that doubled Khovanov homology possesses a perturbation analogous to that defined by Lee in the classical case and define a \emph{doubled Rasmussen invariant}. This invariant is used to obtain various cobordism obstructions; in particular it is an obstruction to sliceness. Finally, we show that the doubled Rasmussen invariant contains the odd writhe of a virtual knot, and use this to show that knots with non-zero odd writhe are not slice.
\end{abstract}

\maketitle

\section{Introduction}\label{Sec:intro}
\subsection{Statement of results}
\label{Subsec:summary}
This paper defines and investigates the properties of a homology theory of virtual links, the titular \emph{doubled Khovanov homology}. For a virtual link \( L \) we denote by \( \dkh ( L ) \) its doubled Khovanov homology, a bigraded finitely generated Abelian group. Below are two examples of the doubled Khovanov homologies of links, split horizontally by the first (homological) grading and vertically by second (quantum) grading; for more detail see \Cref{Fig:21,Fig:vH}. The position of \( = \) indicates \( 0 \) in the quantum grading, and the right-hand column of the first pair of grids is at homological degree \(0\):

\begin{center}
		\( \dkh \left( ~ \raisebox{-7pt}{\includegraphics[scale=0.2]{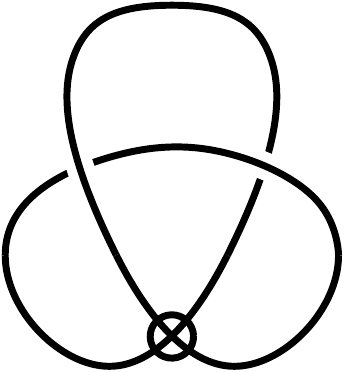}} ~ \right) \quad = \quad	\raisebox{-101.5pt}{
		\begin{tikzpicture}[scale=0.455]		
			%Doubled
			\node[] (-2-3)at(-2,-3) {$ \Z $} ;
			\node[] (-2-2)at(-2,-2) {$ \Z $} ;
			\node[] (-2-1)at(-2,-1) {$ \Z $} ;
			\node[] (-20)at(-2,-0) {$ \Z $} ;

			\node[] (0-1)at(0,-1) {$ \Z $} ;
		
			\node[] (21)at(2,1) {$ {\Z}_2 $} ;
			\node[] (23)at(2,3) {$ \hspace*{-5pt}\Z $} ;

			%Dye
			\node[] (-20)at(6,-2) {$ \Z $} ;
			
			\node[] (02)at(8,0) {$ {\Z}_2 $} ;
			\node[] (04)at(8,2) {$ \hspace*{-5pt}\Z $} ;
			
			\node[] (23)at(10,1) {$ \Z $} ;
			\node[] (25)at(10,3) {$ \Z $} ;
		
			%Axes
			\draw[black, -] (-1,6) -- (-1,-4) ;
			\draw[black, -] (1,6) -- (1,-4) ;
			\draw[black, -] (-3,-2.5) -- (3,-2.5) ;
			\draw[black, -] (-3,-1.5) -- (3,-1.5) ;
			\draw[black, -] (-3,-0.5) -- (3,-0.5) ;
			\draw[black, -] (-3,0.5) -- (3,0.5) ;
			\draw[black, -] (-3,1.5) -- (3,1.5) ;
			\draw[black, -] (-3,2.5) -- (3,2.5) ;
			\draw[black, -] (-3,3.5) -- (3,3.5) ;
			\draw[black, -] (-3,4.5) -- (3,4.5) ;
			\draw[black, -] (-3,5.5) -- (3,5.5) ;
			
			\draw[black, -] (7,6) -- (7,-4) ;
			\draw[black, -] (9,6) -- (9,-4) ;
			\draw[black, -] (5,-2.5) -- (11,-2.5) ;
			\draw[black, -] (5,-1.5) -- (11,-1.5) ;
			\draw[black, -] (5,-0.5) -- (11,-0.5) ;
			\draw[black, -] (5,0.5) -- (11,0.5) ;
			\draw[black, -] (5,1.5) -- (11,1.5) ;
			\draw[black, -] (5,2.5) -- (11,2.5) ;
			\draw[black, -] (5,3.5) -- (11,3.5) ;
			\draw[black, -] (5,4.5) -- (11,4.5) ;
			\draw[black, -] (5,5.5) -- (11,5.5) ;
		
			\end{tikzpicture}}
			\quad = \quad vKh \left( ~ \raisebox{-7pt}{\includegraphics[scale=0.2]{knot21.pdf}} ~ \right)
		\)
\end{center}

\vspace*{5pt}

\begin{center}
		\( \dkh \left( ~ \raisebox{-7pt}{\includegraphics[scale=0.25]{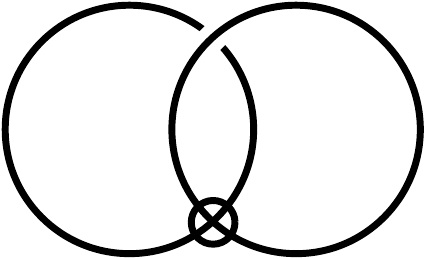}} ~ \right) \quad = \quad	\raisebox{-49.5pt}{
		\begin{tikzpicture}[scale=0.455]		
			%Doubled
			\node[] (0-3)at(0,-3) {$ \Z $} ;

			\node[] (2-1)at(2,-1) {$ {\Z}_2 $} ;
			\node[] (21)at(2,1) {$ \hspace*{-5pt}\Z $} ;
			
			%Dye
			\node[] (6-1)at(6,-1) {$ \Z $} ;
			\node[] (6-3)at(6,-3) {$ \Z $} ;
			
			\node[] (8-2)at(8,-2) {$ \Z $} ;
			\node[] (80)at(8,0) {$ \Z $} ;
			
			%Axes
			\draw[black, -] (7,2) -- (7,-4) ;
			\draw[black, -] (1,2) -- (1,-4) ;
			
			\draw[black, -] (-1,-3.5) -- (3,-3.5) ;
			\draw[black, -] (-1,-2.5) -- (3,-2.5) ;
			\draw[black, -] (-1,-1.5) -- (3,-1.5) ;
			\draw[black, -] (-1,-0.5) -- (3,-0.5) ;
			\draw[black, -] (-1,0.5) -- (3,0.5) ;
			\draw[black, -] (-1,1.5) -- (3,1.5) ;
			\draw[black, -] (5,-3.5) -- (9,-3.5) ;
			\draw[black, -] (5,-2.5) -- (9,-2.5) ;
			\draw[black, -] (5,-1.5) -- (9,-1.5) ;
			\draw[black, -] (5,-0.5) -- (9,-0.5) ;
			\draw[black, -] (5,0.5) -- (9,0.5) ;
			\draw[black, -] (5,1.5) -- (9,1.5) ;
			
		\end{tikzpicture}}
		\quad = \quad vKh \left( ~ \raisebox{-7pt}{\includegraphics[scale=0.25]{virtualhopflink.pdf}} ~ \right)
		\)
\end{center}

\noindent Also depicted is the homology of virtual links first defined by Manturov \cite{Manturov2006} and reformulated by Dye, Kaestner, and Kauffman \cite{Dye2014}, denoted by \( vKh \). One observes that, while the groups assigned to the links by each theory are not disjoint, they differ substantially. Specifically, we see that \( vKh \left( \raisebox{-4pt}{\includegraphics[scale=0.15]{knot21.pdf}} \right) \) and \( vKh \left( \raisebox{-4pt}{\includegraphics[scale=0.2]{virtualhopflink.pdf}} \right) \) both contain a \( \Z^{\oplus 2} \) term for each component of the argument, and that \( vKh \left( \raisebox{-4pt}{\includegraphics[scale=0.15]{knot21.pdf}} \right) \) also contains the knight's move familiar from classical Khovanov homology \cite{Bar-Natan2002}. In contrast \( \dkh \left( \raisebox{-4pt}{\includegraphics[scale=0.15]{knot21.pdf}} \right) \) contains a knight's move and a \( \Z^{\oplus 4} \) term, whereas \( \dkh \left( \raisebox{-4pt}{\includegraphics[scale=0.2]{virtualhopflink.pdf}} \right) \) contains only a single knight's move.

Doubled Khovanov homology can sometimes detect non-classicality of a virtual link.
\begin{theorem*}[\Cref{Cor:nonclass} of \Cref{Subsec:nonclassical}]
	 Let \( L \) be a virtual link. If
	\begin{equation*}
	\dkh ( L ) \neq  G \oplus G \lbrace -1 \rbrace
	\end{equation*}
	for \( G \) a non-trivial bigraded Abelian group, then \( L \) is non-classical.
\end{theorem*}

The graded Euler characteristic of the theory contains no new information.

\begin{theorem*}
	Let \( L \) be a virtual link. Denote by \( \chi_q \) the graded Euler characteristic of \( \dkh ( L ) \) with respect to the quantum grading. Then \( \chi_q = (1 + q^{-1} ) V_L ( q ) \), for \( V_L (q) \) the Jones polynomial of \(L\).
\end{theorem*}

The connect sum operation on virtual knots exhibits more complicated behaviour than that of the classical case: the result of a connect sum between two virtual knots depends on both the diagrams used and the site at which the connect sum is conducted. Indeed, there are multiple inequivalent virtual knots which can be obtained as connect sums of a fixed pair of virtual knots. A surprising consequence of this that there are non-trivial virtual knots which can be obtained as a connect sum of a pair of unknots. Doubled Khovanov homology yields a condition met by such knots.

\begin{theorem*}[\Cref{Thm:unknotcondition} of \Cref{Subsec:trivialdiagrams}]
	 Let \( K \) be a virtual knot which is a connect sum of two trivial knots. Then \( \dkh ( K ) = \dkh \left( \raisebox{-1.75pt}{\includegraphics[scale=0.3]{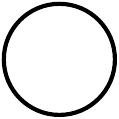}} \right)\). 
\end{theorem*}

Further, there is a perturbation of doubled Khovanov homology akin to Lee's perturbation of Khovanov homology; we denote it by \( \dkh ' ( L ) \) and refer to it as \emph{doubled Lee homology}. Unlike the classical case, however, doubled Lee homology vanishes for certain links. We show this in two steps. Firstly, we prove that the rank of doubled Lee homology behaves analogously to that of classical Lee homology.

\begin{theorem*}[\Cref{Thm:leerank} of \Cref{Subsec:dleedef}]
	 Given a virtual link \( L \)
	\begin{equation*}
	\text{rank} \left( \dkh' ( L ) \right) = 2 \left| \left\{ \text{alternately coloured smoothings of}~ L \right\} \right |.
	\end{equation*}
\end{theorem*}

Secondly, in \Cref{Thm:acs} of \Cref{Subsec:dleedef}, we determine the number of alternately coloured smoothings of a virtual link. In abbreviated form, \Cref{Thm:acs} states that a virtual link \( L \) has either \( 2^{|L|} \) or \( 0 \) alternately coloured smoothings, and that one can determine which case holds via a simple check on a (Gauss diagram of a) diagram of \( L \). This explains why \( \dkh \left( \raisebox{-4pt}{\includegraphics[scale=0.2]{virtualhopflink.pdf}} \right) \) is a single knight's move: a knight's move cancels when we pass to doubled Lee homology and \raisebox{-4pt}{\includegraphics[scale=0.2]{virtualhopflink.pdf}} has no alternately coloured smoothings.

Kauffman related alternately coloured smoothings of virtual link diagrams to perfect matchings of \(3\)-valent graphs \cite{Kauffman2004b}, and using that correspondence we observe that doubled Lee homology yields the following equivalent to the Four Colour Theorem. (A graph is \emph{bridgeless} if it does not possess an edge the removal of which increases the number of connected components of the graph.)

\begin{theorem*}
	Let \( G \) be a planar, bridgeless, \(3\)-valent graph and \( \mathcal{E} \) a perfect matching of \( G \). Associated to the pair \( (G, \mathcal{E} ) \) is a family of virtual link diagrams. Denote a member of this family by \( D (G, \mathcal{E} ) \). Then there exists a perfect matching \( \mathcal{E} \) such that 
	\begin{equation*}
	\dkh ' ( D (G, \mathcal{E} ) ) \neq 0
	\end{equation*}
	for all \( D (G, \mathcal{E} ) \).
\end{theorem*}

Doubled Lee homology cannot vanish for virtual knots, however; we show that a virtual knot has exactly \( 2 \) alternately coloured smoothings, so that its homology is of rank \(4\). In \Cref{Sec:Ras} we show that the information contained in \( \dkh ' ( K ) \) is equivalent to a pair of integers, denoted \( \mathbbm{s} ( K ) = ( s_1 ( K ), s_2 ( K ) ) \), and referred to as the \emph{doubled Rasmussen invariant}; \(s_1 (K) \) contains information regarding the quantum grading, \( s_2(K) \) the homological grading. Using \( \mathbbm{s} ( K ) \) we are able to give the following obstructions to the existence of various kinds of cobordisms.

\begin{theorem*}[\Cref{Thm:vgenusbound2} of \Cref{Subsec:genusbounds}]
	 Let \( K_1 \) and \( K_2 \) be a pair of virtual knots with \( s_2 ( K_1 ) = s_2 ( K_2 ) \), and \( S \) be a certain type of cobordism from between them such that every link appearing in \( S \) has a generator in homological degree \( s_2 ( K ) \). Then
	\begin{equation*}
	\dfrac{| s_1 ( K_1 ) - s_1 ( K_2 ) |}{2} \leq g ( S ).
	\end{equation*}
\end{theorem*}

\begin{theorem*}[\Cref{Thm:merging} of \Cref{Subsec:combining}]
	 Let  \( L \) be a virtual link of \( | L | \) components. Further, let \( S \) be a connected concordance between \( L \) and a virtual knot \( K \) such that \( {\dkh '}_{s_2 ( K ) } ( L ) \neq 0 \). Let \( M(L) \) be the maximum non-trivial quantum degree of elements \(x \in \dkh ' ( L ) \) such that \( \phi_S ( x ) \neq 0 \). Then
	 \begin{equation*}
	 M ( L ) \leq s_1( K ) + | L |.
	 \end{equation*}
\end{theorem*}

Both components of the doubled Rasmussen invariant are concordance invariants and obstructions to sliceness; in \Cref{Subsec:cobordisms,Subsec:obstructions} we use the functorial nature of doubled Lee homology to show this. In addition, the homological degree information contained in the invariant is equivalent to the odd writhe, so that we are able to show that this well known invariant is also an obstruction to sliceness.

\begin{theorem*}[\Cref{Prop:s2odd} of \Cref{Subsec:rasoddwrithe}]
	Let \( K \) be a virtual knot. Then \( s_2 ( K ) = J ( K ) \), where \( J ( K ) \) is the odd writhe of \( K \).
\end{theorem*}

\begin{theorem*}[\Cref{Thm:owritheslice} of \Cref{Subsubsec:oddwritheslice}]
	Let \( K \) be a virtual knot and \( J ( K ) \) its odd writhe. If \( J ( K ) \neq 0 \) then \( K \) is not slice.
\end{theorem*}

\begin{theorem*}[\Cref{Thm:conceobs} of \Cref{Subsec:genusbounds}]
	Let \( K \) and \( K' \) be virtual knots such that \( s_2 ( K ) = s_2 ( K' ) \). If \( s_1 ( K ) \neq s_1 ( K' ) \) then \(K\) and \(K'\) are not concordant.
\end{theorem*}

Finally, using the above results, we show that there exist virtual knots which are not concordant to any classical knots.

\begin{theorem*}[\Cref{Cor:classicalconcordance} of \Cref{Subsubsec:oddwritheslice}]
	Let \( K \) be a virtual knot. If \( J ( K ) \neq 0 \) then \( K \) is not concordant to a classical knot. 
\end{theorem*}

\subsection{Extending Khovanov homology}
The first successful extension of Khovanov homology to virtual links was produced by Manturov \cite{Manturov2006}, as mentioned above. His work was reformulated by Dye, Kaestner, and Kauffman in order to define a virtual Rasmussen invariant \cite{Dye2014}. Tubbenhauer has also developed a virtual Khovanov homology using non-orientable cobordisms \cite{Tubbenhauer2014a}. Doubled Khovanov homology is as an alternative extension of Khovanov homology to virtual links.

Any extension of Khovanov homology to virtual links must deal with the fundamental problem presented by the \textit{single cycle smoothing}, also known as the \textit{one-to-one bifurcation}. This is depicted in \Cref{Fig:121}: altering the resolution of a crossing no longer either splits one cycle or merges two cycles, but can in fact take one cycle to one cycle. The realisation of this as a cobordism between smoothings is a once-punctured M\"{o}bius band. How does one associate an algebraic map, \( \eta \), to this? Looking at the quantum grading (where the module associated to one cycle is \( \mathcal{A} = \langle v_+, v_- \rangle \){)} we notice that
\begin{center}
	\begin{tikzpicture}[scale=0.8,
	roundnode/.style={}]
	
	\node[roundnode] (s0)at (-1,0)  {
		\(\begin{matrix}
		0 \\
		v_+ \\
		0 \\
		v_-
		\end{matrix}\)};
	
	\node[roundnode] (s1)at (1,0)  {
		\(\begin{matrix}
		v_+ \\
		0 \\
		v_- \\
		0
		\end{matrix}\)};
	
	\draw[->,thick] (s0)--(s1) node[above,pos=0.5]{\( \eta \)} ;
	
	\end{tikzpicture}
\end{center}
from which we observe that the map \( \eta : \mathcal{A} \rightarrow \mathcal{A} \) must be the zero map if it is to be grading-preserving (we have arranged the generators vertically by quantum grading). This is the approach taken by Manturov and subsequently Dye et al. 

Another way to solve this problem is to ``double up'' the complex associated to a link diagram in order to plug the gaps in the quantum grading, so that the \( \eta \) map may be non-zero. The notion of ``doubling up'' will be made precise in \Cref{Sec:definition}, but for now let us look at the example of the single cycle smoothing: if we take the direct sum of the standard Khovanov chain complex with itself, but shifted in quantum grading by \( -1 \), we obtain \( \eta : \mathcal{A} \oplus \mathcal{A} \lbrace -1 \rbrace \rightarrow  \mathcal{A} \oplus \mathcal{A} \lbrace -1 \rbrace \), that is
\begin{center}
	\begin{tikzpicture}[scale=0.8,
	roundnode/.style={}]
	
	\node[roundnode] (s0)at (-1,0)  {
		\(\begin{matrix}
		0 \\
		\vup \\
		\vlp \\
		\vum \\
		\vlm
		\end{matrix}\)};
	
	\node[roundnode] (s1)at (1,0)  {
		\(\begin{matrix}
		\vup \\
		\vlp \\
		\vum \\
		\vlm \\
		0
		\end{matrix}\)};
	
	\draw[->,thick] (s0)--(s1) node[above,pos=0.5]{\( \eta \)} ;
	
	\end{tikzpicture}
\end{center}
where \( \mathcal{A} = \langle v^\text{u}_+, v^\text{u}_- \rangle \) and \( \mathcal{A} \lbrace -1 \rbrace = \langle v^\text{l}_+, v^\text{l}_- \rangle \) (u for ``upper'' and l for ``lower'') are graded modules and for \( W \) a graded module \( W_{l-k} = {W \lbrace k \rbrace}_l\). Thus the map associated to the single cycle smoothing may now be non-zero while still degree-preserving. It is this approach which we take in the following work.

\begin{figure}
	\centering
	\begin{tikzpicture}[scale=0.8,
	roundnode/.style={}]
	
	\node[roundnode] (s0)at (-3,0)  {
		\includegraphics[scale=0.7]{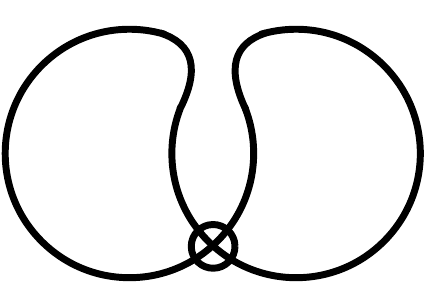}};
	
	\node[roundnode] (s1)at (3,0)  {
		\includegraphics[scale=0.7]{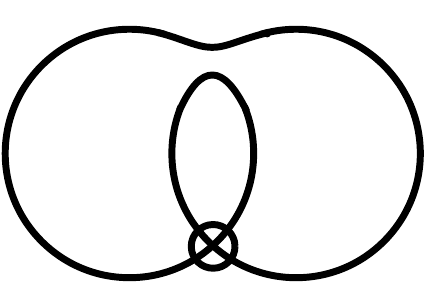}};
	
	\draw[->,thick] (s0)--(s1) node[above,pos=0.5]{\( \eta \)} ;			
	\end{tikzpicture}
	\caption{The single-cycle smoothing.}\label{Fig:121}
\end{figure}

\subsection{Plan of the paper}
In \Cref{Sec:definition} we define the doubled Khovanov homology theory and describe some of its properties: we find the doubled Khovanov homology of classical links, and illustrate a method to produce an infinite number of non-trivial virtual knots with doubled Khovanov homology of the unknot. In \Cref{Sec:leehom} we define a perturbation analogous to Lee homology of classical links and show that, as in the classical case, the rank of this perturbed theory can be computed in terms of alternately coloured smoothings. We then investigate the functorial nature of the perturbed theory. \Cref{Sec:Ras} contains the definition of the doubled Rasmussen invariant and a description of its properties. Finally, in \Cref{Sec:applications} the invariants are put to use, yielding topological applications.

We assume familiarity with classical Khovanov homology and the rudiments of virtual knot theory.

\section{Doubled Khovanov homology}\label{Sec:definition}
\subsection{Definition}

In the tradition of classical Khovanov homology and its descendants \textit{doubled Khovanov homology} assigns to an oriented virtual link diagram a bigraded Abelian group which is the homology of a chain complex; the result is an invariant of the link represented. In contrast to other virtual extensions of Khovanov homology, however, the work of dealing with the single cycle smoothing (see \Cref{Fig:121}) is done in the realm of algebra so that certain verifications require no new technology to complete (c.f. with the \textit{order} construction used in \cite{Dye2014}). 
\begin{definition}[Doubled Khovanov complex]
	\label{Def:cdkh}
	Let \( L \) be an oriented virtual link diagram with \( n_+ \) positive classical crossings and \( n_- \) negative classical crossings. Form the cube of smoothings associated to \( L \) in the standard manner by resolving classical crossings and leaving virtual crossings unchanged -- see the example given in \Cref{Fig:cube}.
	
	Let \( \mathcal{A} = \mathcal{R} [X]/X^2 = \langle v_-, v_+ \rangle \) (under the identification \( X=v_- \), \( 1=v_+ \)) where \( \mathcal{R} \) is either \( \Q \) or \( \Z \). Form a chain complex by associating to a smoothing consisting of \( m \) cycles (that is, \( m \) copies of \( S^1 \) immersed in the plane) a vector space in the following way
	\begin{equation}
	\label{Eq:tqft}
	\underset{ 1 \leq i \leq m}{\bigsqcup} ~ S^1_i \longmapsto \left( \mathcal{A}^{\otimes m} \right) \oplus \left( \mathcal{A}^{\otimes m} \right) \lbrace -1 \rbrace .
	\end{equation}
	We refer to the unshifted (shifted) summand as the \emph{upper} \emph{(lower)} summand and denote elements in the upper summand by a superscript {u} and those in the lower summand by a superscript {l}. We also suppress tensor products, concatenating them into one subscript e.g.
	\[
	v^{\text{\emph{u}}}_{--+-} \coloneqq \left( v_- \otimes v_- \otimes v_+ \otimes v_- \right)^{\text{\emph{u}}} \in \mathcal{A}^{\otimes 4}
	\]
	or
	\[
	v^{\text{\emph{l}}}_{++} \coloneqq \left( v_+ \otimes v_+ \right)^{\text{\emph{l}}} \in \left( \mathcal{A}^{\otimes 2} \right) \lbrace -1 \rbrace .
	\]
	
	\begin{figure}
		\begin{tikzpicture}[scale=0.6,
		roundnode/.style={}]
		
		\node[roundnode] (s4)at (-13,0)  {\includegraphics[scale=0.45]{knot21.pdf}
		};
		
		\node[roundnode] (s0)at (-7,0)  {\includegraphics[scale=0.45]{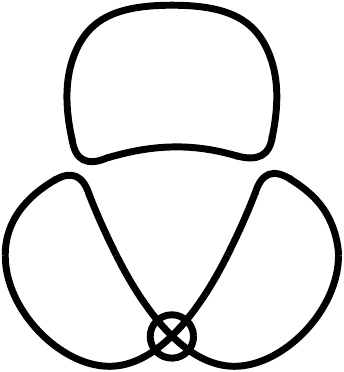}
		};
		
		\node[roundnode] (s1)at (-2,2)  {\includegraphics[scale=0.45]{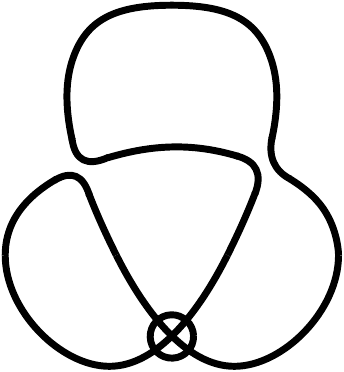}
		};
		
		\node[roundnode] (s2)at (-2,-2)  {\includegraphics[scale=0.45]{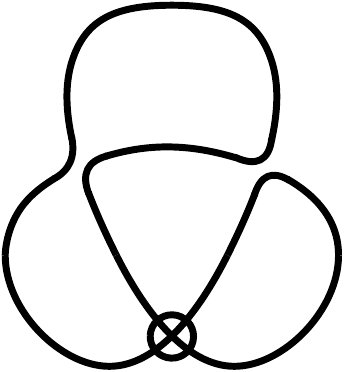}
		};
		
		\node[roundnode] (s3)at (3,0)  {\includegraphics[scale=0.45]{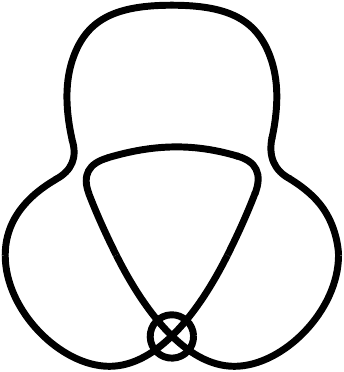}
		};
		
		\draw[->,double] (s4)--(s0) ;
		
		\draw[->,thick] (s0)--(s1) node[above left,pos=0.6]{\( m \)} ;
		
		\draw[->,thick] (s0)--(s2) node[below left,pos=0.6]{\( m \)} ;
		
		\draw[->,thick] (s1)--(s3) node[above right,pos=0.4]{\( \eta \)} ;
		
		\draw[->,thick] (s2)--(s3) node[below right,pos=0.3]{\( -\eta \)} ;
		\end{tikzpicture}
		\caption{The cube of smoothings associated to the virtual knot diagram depicted on the left of the figure.}
		\label{Fig:cube}
	\end{figure}

	\begin{figure}
		\begin{tikzpicture}[scale=0.6,
		roundnode/.style={}]
		\node[roundnode] (s5)at (-11,-7) {\( \cdkh \left( \raisebox{-10pt}{\includegraphics[scale=0.25]{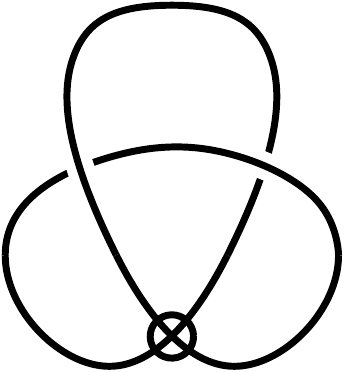}} \right) ~ = ~ \begin{matrix}
			\mathcal{A}^{\otimes 2} \\
			\oplus \\
			\mathcal{A}^{\otimes 2} \lbrace -1 \rbrace \\
			\end{matrix} \)
		};
		
		\node[roundnode] (s6)at (-1.5,-7) {\( \begin{matrix}
			\mathcal{A} \\
			\oplus \\
			\mathcal{A} \lbrace -1 \rbrace \\
			\oplus \\
			\mathcal{A} \\
			\oplus \\
			\mathcal{A} \lbrace -1 \rbrace
			\end{matrix} \)
		};
		
		\node[roundnode] (s7)at (5.5,-7) {\( \begin{matrix}
			\mathcal{A} \\
			\oplus \\
			\mathcal{A} \lbrace -1 \rbrace \\
			\end{matrix} \)
		};
		
		\node[roundnode] (s8)at(-8.5,-10.5) {\( -2 \)};
		
		\node[roundnode] (s9)at(-1,-10.5) {\( -1 \)};
		
		\node[roundnode] (s10)at(5.5,-10.5) {\( 0 \)};
		
		\draw[->,thick] (s5)--(s6) node[above,pos=0.5]{\( d_{-2} =  \begin{pmatrix}
			m \\
			m
			\end{pmatrix} \)
		};
		
		\draw[->,thick] (s6)--(s7) node[above,pos=0.5]{\( d_{-1} =  \left( \eta, - \eta \right) \)
		};
		\end{tikzpicture}
		\caption{The chain complex associated to the cube of smoothings depicted in \Cref{Fig:cube} (homological degree is denoted beneath the chain groups).}
		\label{Fig:algcomp}
	\end{figure}
	
	The components of the differential are built in the standard way as matrices with entries the maps \( \Delta \), \( m \), and \( \eta \), whose positions are read off from the cube of smoothings. The \( \Delta \) and \( m \) maps are given by
	\begin{equation}
	\label{Eq:diffcomp}
	\begin{aligned}
	m( v^{\text{\emph{u/l}}}_+ \otimes v^{\text{\emph{u/l}}}_+ ) & = v^{\text{\emph{u/l}}}_+ \qquad &\Delta ( v^{\text{\emph{u/l}}}_+ ) & = v^{\text{\emph{u/l}}}_+ \otimes v^{\text{\emph{u/l}}}_-  + v^{\text{\emph{u/l}}}_- \otimes v^{\text{\emph{u/l}}}_+  \\
	m( v^{\text{\emph{u/l}}}_+ \otimes v^{\text{\emph{u/l}}}_- ) &= m( v^{\text{\emph{u/l}}}_- \otimes v^{\text{\emph{u/l}}}_+ ) = v^{\text{\emph{u/l}}}_- \qquad &\Delta  ( v^{\text{\emph{u/l}}}_- ) & = v^{\text{\emph{u/l}}}_- \otimes v^{\text{\emph{u/l}}}_- \\
	m(v^{\text{\emph{u/l}}}_- \otimes v^{\text{\emph{u/l}}}_- ) & = 0 & & 
	\end{aligned}
	\end{equation}
	(so that they do not map between the upper and lower summands). The map associated to the single cycle smoothing as in \Cref{Fig:121} is given by
	\begin{equation}
	\label{Eq:etamap}
	\begin{aligned}
	\eta ( v^{\text{\emph{u}}}_+ ) & = v^{\text{\emph{l}}}_+ \qquad & \eta ( v^{\text{\emph{l}}}_+ ) & = 2 v^{\text{\emph{u}}}_- \\
	\eta ( v^{\text{\emph{u}}}_- ) & = v^{\text{\emph{l}}}_- \qquad & \eta ( v^{\text{\emph{l}}}_- ) & = 0.
	\end{aligned}
	\end{equation}
	The effect of the \( \eta \) map on tensor products is (possibly) to alter the superscript of entire string and the subscript of the tensorand in question. For example, if the cycle undergoing the single cycle smoothing is corresponds to the second tensor factor
	\begin{equation*}
	\begin{aligned}
	\eta ( v^{\text{\emph{u}}}_{-+-} ) &= v^{\text{\emph{u}}}_{---} \\
	\eta ( v^{\text{\emph{l}}}_{++-} ) &= 2v^{\text{\emph{u}}}_{+--}.
	\end{aligned}
	\end{equation*}
	Any assignment of signs to the maps within the cube of smoothings which yields anticommutative faces produces isomorphic chain complexes.
	
	Let \( C_i \) denote the direct sum of the vector spaces assigned to the smoothings with exactly \( i \) \( 1 \)-resolutions.	Define the chain spaces of the \emph{doubled Khovanov complex} to be
	\begin{equation}
	\label{Eq:doubled}
	\cdkh_i ( L ) = C_i [-n_-] \lbrace n_+ - 2 n_- \rbrace
	\end{equation}
	(where \( [-n_-] \) denotes a shift in homological degree). An example of such a chain complex is given in \Cref{Fig:algcomp}.\CloseDef
\end{definition}

\begin{remark}
	The map given in \Cref{Eq:etamap} is not an \( \mathcal{R} \)-module map, so that \( \left( \mathcal{A}, m, \Delta, \eta \right) \) is not an extended Frobenius algebra in the sense of \cite{Turaev2006}, and doubled Khovanov homology seemingly cannot be interpreted as an unoriented topological field theory. (Also, doubled Lee homology, as defined in \Cref{Sec:leehom}, does not satisfy the multiplicativity axiom of an unoriented topological field theory.)
\end{remark}

\begin{proposition}
	\label{Prop:chaincom}
	Equipped with the differential given by matrices of maps as described in \Cref{Def:cdkh} \( \cdkh ( L ) \) is a chain complex.
\end{proposition}
\begin{proof}
	It is enough to verify the commutativity of the faces
	
	\begin{center}
		\begin{tikzpicture}[scale=0.6,
		roundnode/.style={}]
		
		\node[roundnode] (s0)at (-2,0)  {
		};
		
		\node[roundnode] (s1)at (0,2)  {
		};
		
		\node[roundnode] (s2)at (0,-2)  {
		};
		
		\node[roundnode] (s3)at (2,0)  {
		};

		\draw[->,thick] (s0)--(s1) node[above left,pos=0.6]{\( \eta \)} ;
		
		\draw[->,thick] (s0)--(s2) node[below left,pos=0.6]{\( \Delta \)} ;
		
		\draw[->,thick] (s1)--(s3) node[above right,pos=0.4]{\( \eta \)} ;
		
		\draw[->,thick] (s2)--(s3) node[below right,pos=0.4]{\( m \)} ;
		\end{tikzpicture}
		\qquad
		\begin{tikzpicture}[scale=0.6,
		roundnode/.style={}]
		
		\node[roundnode] (s0)at (-2,0)  {
		};
		
		\node[roundnode] (s1)at (0,2)  {
		};
		
		\node[roundnode] (s2)at (0,-2)  {
		};
		
		\node[roundnode] (s3)at (2,0)  {
		};

		\draw[->,thick] (s0)--(s1) node[above left,pos=0.6]{\( \eta \)} ;
		
		\draw[->,thick] (s0)--(s2) node[below left,pos=0.6]{\( \Delta \)} ;
		
		\draw[->,thick] (s1)--(s3) node[above right,pos=0.4]{\( \Delta \)} ;
		
		\draw[->,thick] (s2)--(s3) node[below right,pos=0.4]{\( \eta \)} ;
		\end{tikzpicture}
		\qquad
		\begin{tikzpicture}[scale=0.6,
		roundnode/.style={}]
		
		\node[roundnode] (s0)at (-2,0)  {
		};
		
		\node[roundnode] (s1)at (0,2)  {
		};
		
		\node[roundnode] (s2)at (0,-2)  {
		};
		
		\node[roundnode] (s3)at (2,0)  {
		};

		\draw[->,thick] (s0)--(s1) node[above left,pos=0.6]{\( \eta \)} ;
		
		\draw[->,thick] (s0)--(s2) node[below left,pos=0.6]{\( m \)} ;
		
		\draw[->,thick] (s1)--(s3) node[above right,pos=0.4]{\( m \)} ;
		
		\draw[->,thick] (s2)--(s3) node[below right,pos=0.4]{\( \eta \)} ;
		\end{tikzpicture}
	\end{center}
	as the face
	\begin{center}
		\begin{tikzpicture}[scale=0.6,
		roundnode/.style={}]
		
		\node[roundnode] (s0)at (-2,0)  {
		};
		
		\node[roundnode] (s1)at (0,2)  {
		};
		
		\node[roundnode] (s2)at (0,-2)  {
		};
		
		\node[roundnode] (s3)at (2,0)  {
		};

		\draw[->,thick] (s0)--(s1) node[above left,pos=0.6]{\( \eta \)} ;
		
		\draw[->,thick] (s0)--(s2) node[below left,pos=0.6]{\( m \)} ;
		
		\draw[->,thick] (s1)--(s3) node[above right,pos=0.4]{\( \eta \)} ;
		
		\draw[->,thick] (s2)--(s3) node[below right,pos=0.4]{\( \Delta \)} ;
		\end{tikzpicture}
	\end{center}
	cannot occur. We leave the algebra to the reader and note that, as in the classical case, sprinkling signs appropriately yields a chain complex.
\end{proof}

\begin{theorem}
	\label{Thm:invariance}
	Given an oriented virtual link diagram \( D \) the chain homotopy equivalence class of \( \cdkh ( D ) \) is an invariant of the oriented link represented by \( D \).
	The homology of \( \cdkh ( D ) \), denoted \( \dkh ( D ) \), is therefore also an invariant of the link represented by \( D \).
\end{theorem}
\begin{proof}
	We are required to construct chain homotopy equivalences for each of the virtual Reidemeister moves. It is readily observed that if two diagrams \( D_1 \) and \( D_2 \) are related by a finite sequence of the purely virtual moves and mixed move (depicted in \Cref{Fig:vrm}) then \( \cdkh ( D_1 ) = \cdkh ( D_2 ) \) as these moves do not alter the number of cycles in a smoothing nor the incoming and outcoming maps.
	\begin{figure}		
		\includegraphics[scale=0.60]{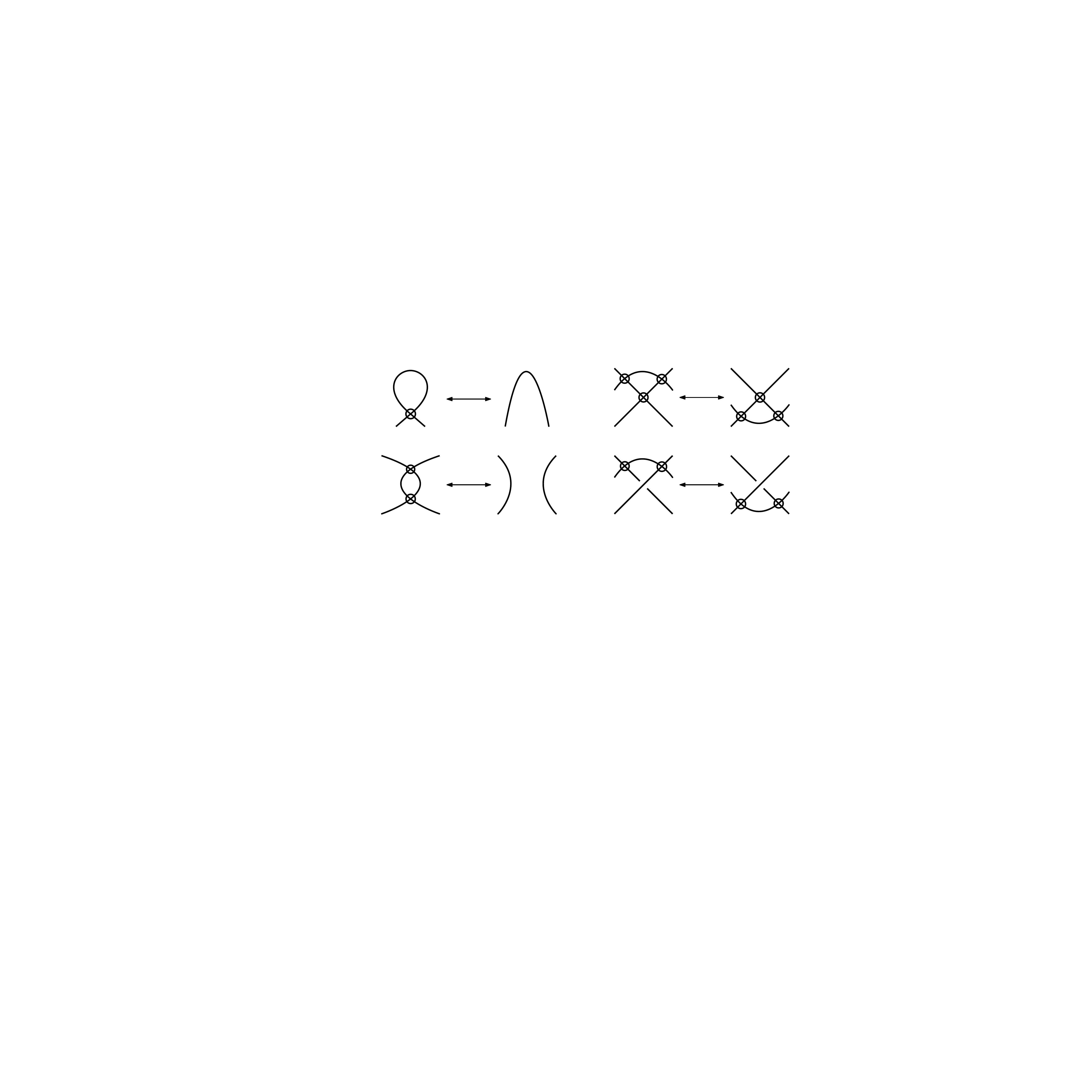}
		\caption{The purely virtual Reidemeister moves and the mixed move (bottom right of the figure).}
		\label{Fig:vrm}
	\end{figure}
	Concerning the classical moves, we follow Bar-Natan \cite{Bar-Natan2002}, using \cite[Lemma \(3.7\)]{Bar-Natan2002} and Gauss elimination (specifically, \cite[Lemma \( 3. 2\)]{Bar-Natan2014}). We leave the details to the reader.
\end{proof}

The homology of the complex given in \Cref{Fig:algcomp} is depicted in \Cref{Fig:21}.

\subsection{Detection of non-classicality}\label{Subsec:nonclassical} We say that a virtual link is \emph{non-classical} if all diagrams representing it have at least one virtual crossing. Conversely, we say that a virtual link is \emph{classical} if it has a diagram with no virtual crossings. Doubled Khovanov homology can sometimes be used to detect non-classicality.

Consider the complex associated to the classical diagram of the unknot given in \Cref{Fig:cdkhclass}: the reader notices immediately that not only do the chain spaces decompose as direct sums, the entire complex does also (as there are no \( \eta \) maps). That is
\begin{equation}
\cdkh \left( \raisebox{-10pt}{\includegraphics[scale=0.25]{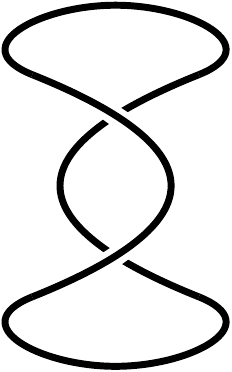}} \right) = CKh \left( \raisebox{-10pt}{\includegraphics[scale=0.25]{unknot.pdf}} \right) \oplus CKh \left( \raisebox{-10pt}{\includegraphics[scale=0.25]{unknot.pdf}} \right) \lbrace -1 \rbrace
\label{Eq:classhom}
\end{equation}
where \( CKh ( D ) \) denotes the classical Khovanov complex of a diagram \( D \). This motivates the following proposition.
\begin{figure}
	\begin{tikzpicture}[scale=0.6,
	roundnode/.style={}]
	\node[roundnode] (s5)at (-11,-7) {\( \cdkh \left( \raisebox{-10pt}{\includegraphics[scale=0.25]{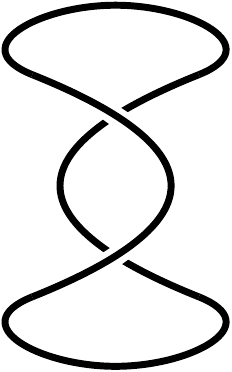}} \right) ~ = ~ \begin{matrix}
		\mathcal{A}^{\otimes 2} \\
		\oplus \\
		\mathcal{A}^{\otimes 2} \lbrace -1 \rbrace \\
		\end{matrix} \)
	};
	
	\node[roundnode] (s6)at (-2,-7) {\( \begin{matrix}
		\mathcal{A}^{\otimes 3} \\
		\oplus \\
		\mathcal{A}^{\otimes 3} \lbrace -1 \rbrace \\
		\oplus \\
		\mathcal{A} \\
		\oplus \\
		\mathcal{A} \lbrace -1 \rbrace
		\end{matrix} \)
	};
	
	\node[roundnode] (s7)at (5.5,-7) {\( \begin{matrix}
		\mathcal{A}^{\otimes 2} \\
		\oplus \\
		\mathcal{A}^{\otimes 2} \lbrace -1 \rbrace \\
		\end{matrix} \)
	};
	
	\node[roundnode] (s8)at(-8.5,-10.5) {\( -1 \)};
	
	\node[roundnode] (s9)at(-1,-10.5) {\( -0 \)};
	
	\node[roundnode] (s10)at(5.5,-10.5) {\( 1 \)};
	
	\draw[->,thick] (s5)--(s6) node[above,pos=0.5]{\( d_{-2} =  \begin{pmatrix}
		\Delta \\
		m
		\end{pmatrix} \)
	};
	
	\draw[->,thick] (s6)--(s7) node[above,pos=0.5]{\( d_{-1} =  \left( m, - \Delta \right) \)
	};
	\end{tikzpicture}
	\caption{The doubled Khovanov complex of a classical diagram.}
	\label{Fig:cdkhclass}
\end{figure}

\begin{proposition}
	\label{Prop:classical}
	Let \( L \) be a virtual link. If \( L \) is classical then there exists a diagram of \( L \), denoted \( D \), which has no classical crossings. Then
	\begin{equation*}
	\dkh ( L ) = Kh ( D ) \oplus Kh ( D ) \lbrace -1 \rbrace
	\end{equation*}
	where \( Kh ( D ) \) denotes the standard Khovanov homology of a classical link.
\end{proposition}
\begin{proof}
	This is an obvious consequence of \Cref{Eq:classhom}, which holds for all classical diagrams.
\end{proof}

The contrapositive statement to that of \Cref{Prop:classical} is:
\begin{corollary}
	\label{Cor:nonclass}
	Let \( L \) be a virtual link. If
	\begin{equation}
	\label{Eq:nonclassical}
	\dkh ( L ) \neq  G \oplus G \lbrace -1 \rbrace
	\end{equation}
	for \( G \) a non-trivial bigraded Abelian group, then \( L \) is non-classical.
\end{corollary}

As an example consider the virtual knot \( 2.1 \) in Green's table \cite{Green}, depicted in \Cref{Fig:21}, along with its doubled Khovanov homology, split by homological grading (horizontal axis) and quantum grading (vertical axis).

\begin{figure}
	\( \dkh \left( ~ \raisebox{-32pt}{\includegraphics[scale=0.65]{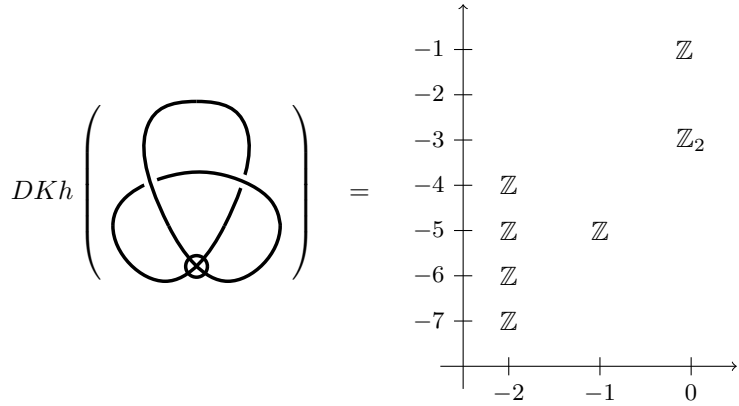}} ~ \right) \quad = \quad	\raisebox{-80pt}{\begin{tikzpicture}[scale=0.6]
		
		%Axes
		\draw[black, ->] (-3.5,-4) -- (3,-4) ;
		\draw[black, ->] (-3,-4.5) -- (-3,4) ;
		
		\draw (-2,-3.8) -- (-2,-4.2) node[below] {\small $-2$} ;
		\draw (0,-3.8) -- (0,-4.2) node[below] {\small $-1$} ;
		\draw (2,-3.8) -- (2,-4.2) node[below] {\small $0$} ;
		
		\draw (-2.8,-3) -- (-3.2,-3) node[left] {\small $-7$} ;
		\draw (-2.8,-2) -- (-3.2,-2) node[left] {\small $-6$} ;
		\draw (-2.8,-1) -- (-3.2,-1) node[left] {\small $-5$} ;
		\draw (-2.8,0) -- (-3.2,0) node[left] {\small $-4$} ;
		\draw (-2.8,1) -- (-3.2,1) node[left] {\small $-3$} ;
		\draw (-2.8,2) -- (-3.2,2) node[left] {\small $-2$} ;
		\draw (-2.8,3) -- (-3.2,3) node[left] {\small $-1$} ;
		
		%Homology
		\node[] (-2-3)at(-2,-3) {$ \Z $} ;
		\node[] (-2-2)at(-2,-2) {$ \Z $} ;
		\node[] (-2-1)at(-2,-1) {$ \Z $} ;
		\node[] (-20)at(-2,-0) {$ \Z $} ;
		
		\node[] (0-1)at(0,-1) {$ \Z $} ;
		
		\node[] (21)at(2,1) {$ {\Z}_2 $} ;
		\node[] (23)at(2,3) {$ \hspace*{-5pt}\Z $} ;
		
		\end{tikzpicture}}
	\)
	\caption{The doubled Khovanov homology of the virtual knot \( 2.1 \).}
	\label{Fig:21}
\end{figure}

Another interesting example is given by the so-called virtual Hopf link, given in \Cref{Fig:vH}; we shall look into it further in \Cref{Sec:leehom}.

\begin{figure}
	\(
	\dkh \left( ~ \raisebox{-20pt}{\includegraphics[scale=0.65]{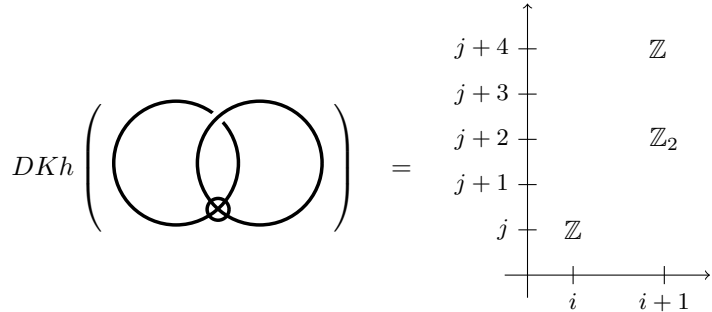}} ~ \right)  \quad = \quad	\raisebox{-55pt}{\begin{tikzpicture}[scale=0.6]
		
		%Axes
		\draw[black, ->] (-1.5,-4) -- (3,-4) ;
		\draw[black, ->] (-1,-4.5) -- (-1,2) ;
		
		\draw (0,-3.8) -- (0,-4.2) node[below] {\small $i$} ;
		\draw (2,-3.8) -- (2,-4.2) node[below] {\small $i+1$} ;
		
		\draw (-0.8,-3) -- (-1.2,-3) node[left] {\small $j$} ;
		\draw (-0.8,-2) -- (-1.2,-2) node[left] {\small $j+1$} ;
		\draw (-0.8,-1) -- (-1.2,-1) node[left] {\small $j+2$} ;
		\draw (-0.8,0) -- (-1.2,0) node[left] {\small $j+3$} ;
		\draw (-0.8,1) -- (-1.2,1) node[left] {\small $j+4$} ;
		
		%Homology
		\node[] (0-3)at(0,-3) {$ \Z $} ;
		\node[] (-2-2)at(2,-1) {$ {\Z}_2 $} ;
		\node[] (-2-2)at(2,1) {$ \hspace*{-5pt}\Z $} ;
		
		\end{tikzpicture}} 
	\)
	\caption{The doubled Khovanov homology of the virtual Hopf link (\( i \) and \(j\) depend on the orientation of the components).}
	\label{Fig:vH}
\end{figure}

The statement within \Cref{Cor:nonclass} cannot be upgraded to an equivalence, however. A counterexample is given by the virtual knot \( 3.7 \) in Green's table, depicted on the right of \Cref{Fig:flank37} (the non-classicality of \( 3.7 \) is demonstrated by its generalised Alexander polynomial \cite{Kauffman}). The cube of smoothings associated to \( 3.7 \) does not contain any \( \eta \) maps, and therefore \( \dkh ( 3.7 ) = G \oplus G \lbrace -1 \rbrace  \) for some non-trivial Abelian group \( G \). In fact, \( \dkh ( 3.7 ) = Kh \left( \raisebox{-1.75pt}{\includegraphics[scale=0.3]{unknotflat.pdf}} \right) \oplus Kh \left( \raisebox{-1.75pt}{\includegraphics[scale=0.3]{unknotflat.pdf}} \right) \lbrace -1 \rbrace = \dkh \left( \raisebox{-1.75pt}{\includegraphics[scale=0.3]{unknotflat.pdf}} \right) \). This follows from the fact that \( 3.7 \) can be obtained from a diagram of the unknot by applying the following move on diagrams

\begin{definition}
	Within an oriented virtual link diagram one may place a virtual crossing on either side of a classical crossing in the following manner
	\begin{center}
		\includegraphics[scale=0.7]{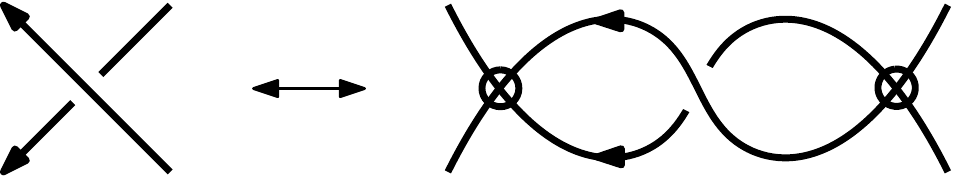}
	\end{center}
	This move is known as \textit{flanking}.\CloseDef
\end{definition}

Flanking is also known \emph{virtualization}, but as there is some confusion in the literature regarding that term we avoid it.

\begin{figure}
	\includegraphics*[scale=0.65]{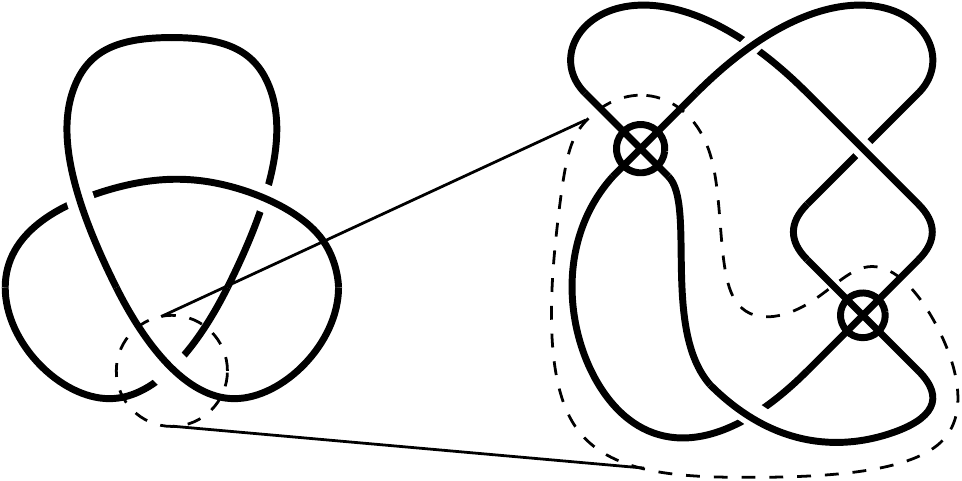}
	\caption{Obtaining virtual knot \( 3.7 \) from the unknot via flanking.}
	\label{Fig:flank37}
\end{figure}

\begin{proposition}
	If a virtual link diagram \( D_2 \) can be obtained from another, \( D_1 \), by a flanking move then \( \cdkh ( D_1) = \cdkh ( D_2 ) \).
\end{proposition}

\begin{proof}
	Let \( D_1 \) and \( D_2 \) be as in the proposition. Consider the tangle diagrams by produced by isolating a neighbourhood of the classical crossing undergoing the flanking move in \(D_1\) and a neighbourhood of the result of the flanking move in \( D_2 \). We construct an identification of the smoothings of \( D_1 \) with those of \( D_2 \) using the smoothings of the tangle diagrams depicted in \Cref{Fig:flanksmooth}: a smoothing of \( D_1 \) must contain either \( 0 ( T_1 ) \) or \( 1 ( T_1 ) \), and we associate to it the smoothing of \( D_2 \) formed by replacing \( 0 ( T_1 ) \) with \( 0 ( T_2 ) \), or \( 1 ( T_1 ) \) with \( 1 ( T_2 ) \). One readily sees that this identification is a bijection which does not change the number of cycles in a smoothing nor how those cycles are linked. Thus the chain spaces of \( \cdkh ( D_1 ) \) and \( \cdkh ( D_2 ) \) are equal, and so are the components of the differential.
\end{proof}

\begin{figure}
	\includegraphics[scale=0.65]{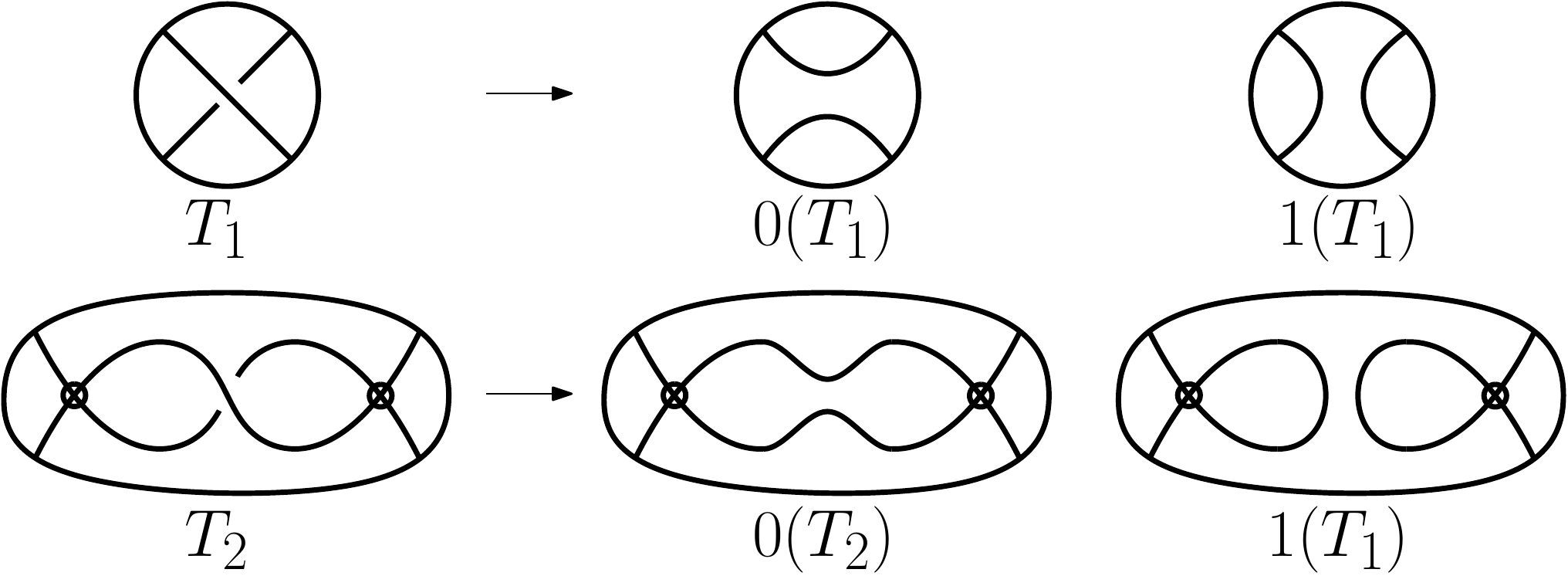}
	\caption{Smoothings of the tangle diagrams related to the flanking move.}
	\label{Fig:flanksmooth}
\end{figure}

\begin{corollary}
	There is an infinite number of non-trivial virtual knots with doubled Khovanov homology equal to that of the unknot.
\end{corollary}

\begin{proof}
	There is an infinite number of non-trivial virtual knot diagrams with unit Jones polynomial, produced via flanking \cite{Dye2005,Kauffman1998,Silver2004}. Each of these knots must also have the doubled Khovanov homology of the unknot.
\end{proof}

\section{Doubled Lee homology}
\label{Sec:leehom}
In \Cref{Subsec:dleedef} we define doubled Lee homology and prove some of its properties, and in \Cref{Subsec:cobordisms} we investigate aspects of the functorial nature of the theory.

\subsection{Definition}
\label{Subsec:dleedef}
The reader may have noticed that there are generators of the homologies depicted in \Cref{Fig:21} and \Cref{Fig:vH} which are \( 4 \) apart in quantum degree. Quantum degree separations of length \( 4 \) are important in classical Khovanov homology; Lee's perturbation of Khovanov homology \cite{Lee2005} is defined by adding to the differential a component of degree \( 4 \). Such a perturbation of doubled Khovanov homology exists also.

\begin{definition}[Doubled Lee homology]
	\label{Def:dkhprime}
	Let \( D \) be an oriented virtual link diagram and \( \cdkh ' ( D ) \) denote the chain complex with the chain spaces of \( \cdkh ( D ) \) but with altered differential, and \( \mathcal{R} = \Q \). The components of this differential are as follows
	\begin{equation*}
	\begin{aligned}
	m'( v^{\text{\emph{u/l}}}_+ \otimes v^{\text{\emph{u/l}}}_+ ) & = v^{\text{\emph{u/l}}}_+ \qquad &\Delta' ( v^{\text{\emph{u/l}}}_+ ) & = v^{\text{\emph{u/l}}}_+ \otimes v^{\text{\emph{u/l}}}_-  + v^{\text{\emph{u/l}}}_- \otimes v^{\text{\emph{u/l}}}_+ \notag \\
	m'( v^{\text{\emph{u/l}}}_+ \otimes v^{\text{\emph{u/l}}}_- ) &= m'( v^{\text{\emph{u/l}}}_- \otimes v^{\text{\emph{u/l}}}_+ ) = v^{\text{\emph{u/l}}}_- \qquad &\Delta'  ( v^{\text{\emph{u/l}}}_- ) & = v^{\text{\emph{u/l}}}_- \otimes v^{\text{\emph{u/l}}}_- + v^{\text{\emph{u/l}}}_+ \otimes v^{\text{\emph{u/l}}}_+ \\
	m'(v^{\text{\emph{u/l}}}_- \otimes v^{\text{\emph{u/l}}}_- ) & = v^{\text{\emph{u/l}}}_+ & & \notag
	\end{aligned}
	\end{equation*}
	and
	\begin{equation*}
	\begin{aligned}
	\eta' ( v^{\text{\emph{u}}}_+ ) & = v^{\text{\emph{u}}}_- \qquad & \eta' ( v^{\text{\emph{l}}}_+ ) & = 2 v^{\text{\emph{u}}}_- \\
	\eta' ( v^{\text{\emph{u}}}_- ) & = v^{\text{\emph{l}}}_- \qquad & \eta' ( v^{\text{\emph{l}}}_- ) & = 2 v^{\text{\emph{u}}}_+.
	\end{aligned}
	\end{equation*}
	The above maps are no longer graded, but filtered (as in the classical case). That \( \cdkh' ( D ) \) is a chain complex is verified as in \Cref{Prop:chaincom}. Setting \( \dkh' ( D ) \) to be the homology of \( \cdkh' ( D ) \), define the doubled Lee homology of \( L \)
	\begin{equation*}
	\dkh' ( L ) \coloneqq \dkh' ( D )
	\end{equation*}
	where \( L \) is the link represented by \( D \).\CloseDef
\end{definition}

\begin{theorem}
	For a virtual link diagram \( D \), \( \dkh ' ( D ) \) is an invariant of the link represented by \( D \).
\end{theorem}

As in the classical case, doubled Khovanov homology and doubled Lee homology are related in the following manner.
\begin{theorem}
	For any virtual link \( L \) there is a spectral sequence with \(E_2\) page \( \dkh ( L ) \) converging to \( \dkh ' ( L ) \).
\end{theorem}

The rank of the classical Lee homology of a link depends only on the number of its components. Precisely, for a classical link \( L_c \)
\begin{equation}
\label{Eq:clee1}
\text{rank} \left( Kh' ( L_c ) \right) = 2^{|L_c|}
\end{equation}
where \( |L_c| \) denotes the number of components of \( L_c \) and \( Kh' ( L_c ) \) its classical Lee homology. In fact, \Cref{Eq:clee1} follows from the following two statements \cite{Bar-Natan2006}:
\begin{equation}
\label{Eq:clee2}
\text{rank} \left( Kh' ( L_c ) \right) = \left| \left\lbrace \text{alternately coloured smoothings of}~ L_c \right\rbrace \right|
\end{equation}
and
\begin{equation}
\label{Eq:clee3}
\left\{ \text{alternately coloured smoothings of}~ L_c \right\} = \left\{ \text{orientations of}~ L_c \right\}
\end{equation}
given the following definition:
\begin{definition}
	\label{Def:ac}
	A smoothing of a virtual link diagram is \emph{alternately coloured} if its cycles are coloured exactly one of two colours in such a way that in a neighbourhood of each classical crossing the two incident arcs are different colours. A smoothing which can be coloured in such a way is known as an \emph{alternately colourable}.\CloseDef
\end{definition}
(Any potential issue raised by the fact that \Cref{Def:ac} regards diagrams while \Cref{Eq:clee2,Eq:clee3} are statements about links is resolved by \Cref{Thm:leerank}, which shows that the number of alternately coloured smoothings is a link invariant.)

In the virtual case we recover \Cref{Eq:clee2} (up to a scalar) but not \Cref{Eq:clee3}.
\begin{theorem}
	\label{Thm:leerank}
	Given a virtual link \( L \)
	\begin{equation*}
	\text{rank} \left( \dkh' ( L ) \right) = 2 \left| \left\{ \text{alternately coloured smoothings of}~ L \right\} \right |.
	\end{equation*}
\end{theorem}
We postpone stating the virtual generalisation of \Cref{Eq:clee3} until we have proved \Cref{Thm:leerank}, for which we require the following analogue of a classical result.
\begin{lemma}
	\label{Lem:action}
	Let \( D \) be a diagram of a virtual link \(L\). There is an action of \( \mathcal{A} \) on \( \cdkh' ( D ) \) which descends to an action on \( \dkh' ( L ) \).
\end{lemma}

\begin{proof}
	Given a virtual link diagram \( D \) define an action of \( \mathcal{A} \) on \( \cdkh' (D) \) in the following manner: mark a point on \( D \) and maintain it across the smoothings of \( D \). The action \( \mathcal{A} \times \cdkh' (D) \rightarrow \cdkh' (D) \) is given by
	\begin{equation*}
	\begin{aligned}
	s \cdot \left( \left( x_1 \otimes x_2 \otimes \ldots \otimes x_n \right)^\text{u} + \left( x_1 \otimes x_2 \otimes \ldots \otimes x_n \right)^\text{l} \right) = &\left( x_1 \otimes x_2 \otimes \ldots \otimes s {\small\cdot} x_i \otimes \ldots \otimes x_n \right)^\text{u} + \\ &\left( x_1 \otimes x_2 \otimes \ldots \otimes s {\small\cdot} x_i \otimes \ldots \otimes x_n \right)^\text{l}
	\end{aligned}
	\end{equation*}
	where the \( i \)-th cycle is marked (component-wise multiplication \( \cdot : \mathcal{A} \times \mathcal{A} \rightarrow \mathcal{A} \) is given by \( m' \)). Clearly this action endows \( \cdkh' ( D ) \) with the structure of an \( \mathcal{A} \)-module. To show that \( \dkh' ( D ) \) is also an \( \mathcal{A} \)-module it suffices to show that the action defined above commutes with the differential. We verify this in the case of \( m' \) and multiplication by \( v_- \), with the marked point on the cycle corresponding to the first tensor factor:
	\begin{center}
		\includegraphics[scale=1]{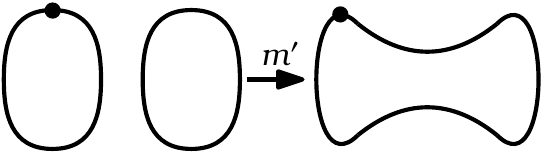}
	\end{center}
	\begin{equation*}
	\begin{aligned}
	v_- \cdot m' ( ( v_+ \otimes v_+ )^\text{u/l} ) = v_- \cdot \vulp &= m' ( ( v_- \otimes v_+ )^\text{u/l} ) = m' ( ((v_- \cdot v_+) \otimes v_+ )^\text{u/l} ) \\
	v_- \cdot m' ( ( v_+ \otimes v_- )^\text{u/l} ) = v_- \cdot \vulm &=  m' ( ( v_- \otimes v_- )^\text{u/l} ) = m' ( ((v_- \cdot v_+) \otimes v_- )^\text{u/l} ) \\
	v_- \cdot m' ( ( v_- \otimes v_+ )^\text{u/l} ) = v_- \cdot \vulm &=  m' ( ( v_- \otimes v_- )^\text{u/l} ) = m' ( ((v_- \cdot v_-) \otimes v_+ )^\text{u/l} ) \\
	v_- \cdot m' ( ( v_- \otimes v_- )^\text{u/l} ) = v_- \cdot \vulp &=  m' ( ( v_- \otimes v_p )^\text{u/l} ) = m' ( ((v_- \cdot v_-) \otimes v_- )^\text{u/l} )
	\end{aligned}
	\end{equation*}
	as required. The other cases are left to the reader.
\end{proof}

Further, we define a new basis for \( \mathcal{A} \): the familiar ``red'' and ``green'' basis first given by Bar-Natan and Morrison.
\begin{definition}
	\label{Def:redgreen}
	Let \( \lbrace r, g \rbrace \) be the basis for \( \mathcal{A} \) where
	\begin{equation*}
	\begin{aligned}
	\text{``red''} &= r = \dfrac{v_+ + v_-}{2} \\
	\text{``green''} &= g = \dfrac{v_+ - v_-}{2}.
	\end{aligned}
	\end{equation*}
	We denote the corresponding generators of \( \mathcal{A} \oplus \mathcal{A}\lbrace -1 \rbrace \) as \( r^\text{u} \), \( r^\text{l} \), \( g^\text{u} \), and \( g^\text{l} \).\CloseDef
\end{definition}
We denote which generator a cycle of a smoothing is labelled with by colouring that cycle either red or green. Thus alternately coloured smoothings are such that given any two cycles which share a crossing (i.e.\ they pass through the same crossing neighbourhood) one is coloured red and the other green.

We shall use the following definition in the remainder of this work.
\begin{definition}
	\label{Def:height}
	Let \( D \) be an oriented virtual link diagram with \( n_- \) negative classical crossings, and \( \mathscr{S} \) an alternately coloured smoothing in which \( m \) classical crossings (positive or negative) resolved into their \(1\)-resolution. Define the \emph{height} of \(\mathscr{S} \) to be \( | \mathscr{S} | \coloneqq m - n_- \). Of course, if \( \sg^{\text{u/l}} \) are the alternately coloured generators assigned to \( \mathscr{S} \) then \( | \mathscr{S} | = i ( \sg^{\text{u/l}} ) \).\CloseDef
\end{definition}

\begin{proof}[Proof of \Cref{Thm:leerank}]
	Let \( D \) be a diagram of \( L \) and \( \mathscr{S} \) be an alternately coloured smoothing of \( D \), with cycles coloured either red or green, and \( \sg^\text{u} \) be the algebraic element given by
	\begin{equation}
	\label{Eq:acgenerators}
	\sg^\text{u} = \bigotimes_{\text{cycles of}~S} \square^{\text{u}}_i
	\end{equation}
	where
	\begin{equation*}
	\square^\text{u}_i = \left\lbrace \begin{array}{c}
	r^u, ~\text{if the}~ i \text{-th cycle is coloured red} \\
	g^u, ~\text{if the}~ i \text{-th cycle is coloured green} 
	\end{array} \right.
	\end{equation*}
	and likewise define \( \sg^\text{l} \), so that to each alternately coloured smoothing we associate two algebraic objects. We refer to such \( \sg^\text{u/l} \)'s as \emph{alternately coloured generators}, a term we justify with the following steps: we shall show that such elements are homologically non-trivial and distinct, and that they do indeed generate \( \dkh'(L) \).
	
	First notice that alternately coloured smoothings have restricted incoming and outgoing differentials: if a smoothing has an \( \eta' \) map either incoming or outgoing then it must have a crossing neighbourhood with only one cycle passing through it. Such a crossing neighbourhood cannot satisfy the alternately coloured condition. Likewise, if a smoothing has an incoming \( m' \) map or an outgoing \( \Delta' \) map it must have a crossing neighbourhood with only one cycle passing through. Thus an alternately coloured smoothing has only incoming \( \Delta' \) maps and outgoing \( m' \) maps and homological non-triviality of the associated \( \sg^\text{u/l} \) is equivalent to \( \sg^\text{u/l} \in \text{ker} ( m' ) \) and \( \sg^\text{u/l} \notin \text{im} ( \Delta' ) \). With respect to the \( \lbrace r, g \rbrace \) basis we have
	\begin{equation}
	\label{Eq:rgmaps}
	\begin{aligned}
	m' ( r \otimes r ) &= r \qquad &\Delta' ( r ) &= 2r \otimes r \\
	m' ( g \otimes g ) &= g \qquad &\Delta' ( g ) &= -2g \otimes g \\
	m' ( r \otimes g ) &= m' ( g \otimes r ) = 0 & &\\
	\end{aligned}
	\end{equation}
	so that clearly \( [ \sg^\text{u/l} ] \neq 0 \in \dkh' ( L ) \).
	
	Let \( \mathscr{S}_1 \) and \( \mathscr{S}_2 \) be two alternately coloured smoothings of \( L \) and \( \sg^\text{u/l}_1 \), \( \sg^\text{u/l}_2 \) their associated alternately coloured generators. Notice that it is possible that \( \mathscr{S}_1 \) and \( \mathscr{S}_2 \) are alternately coloured smoothings associated to the same uncoloured smoothing of \( L \). We shall consider the two cases: \( (i) \) \( \mathscr{S}_1 \) and \( \mathscr{S}_2 \) are not alternately coloured smoothings associated to the same uncoloured smoothing of \( D \) and \( (ii) \) they are.
	
	\noindent\emph{\( (i) \)}: It is possible that \( \mathscr{S}_1 \) and \( \mathscr{S}_2 \) are at different height (that is, they have a different number of \( 1 \)-resolutions). Then \( [ \sg^\text{u/l}_1 ] \neq  [ \sg^\text{u/l}_2 ] \) as they are of differing homological grading. If \( \mathscr{S}_1 \) and \( \mathscr{S}_2 \) are at the same height, \( i \), say, we recall that \( {\cdkh'}_i ( L ) \) is a direct sum of the modules associated to all smoothings of height \(i\) so that \( \sg^\text{u/l}_1 - \sg^\text{u/l}_2 \) can be written
	\begin{equation*}
	\sg^\text{u/l}_1 - \sg^\text{u/l}_2 = \begin{pmatrix}
	0 \\
	\vdots \\
	0 \\
	\sg^\text{u/l}_1 \\
	0 \\
	\vdots \\
	\vdots \\
	0
	\end{pmatrix} - 
	\begin{pmatrix}
	0 \\
	\vdots \\
	\vdots \\
	0 \\
	\sg^\text{u/l}_2 \\
	0 \\
	\vdots \\
	0
	\end{pmatrix} =
	\begin{pmatrix}
	0 \\
	\vdots \\
	0 \\
	\sg^\text{u/l}_1 \\
	0 \\
	\vdots \\
	0 \\
	-\sg^\text{u/l}_2 \\
	0 \\
	\vdots
	\end{pmatrix}
	\end{equation*}
	so that \( \sg^\text{u/l}_1 - \sg^\text{u/l}_2 \notin \text{im}(\Delta') \).
	
	\noindent\emph{\( (ii) \)}: Mark a point on \(L\) such that the cycles of \( \mathscr{S}_1 \) and \( \mathscr{S}_2 \) on which the point lies are opposite colours (such a point always exists as \( \mathscr{S}_1 \neq \mathscr{S}_2 \)), and define the action of \( \mathcal{A} \) as in \Cref{Lem:action}. Notice that \( v_- \cdot r = r \) and \( v_- \cdot g = - g \) so that
	\begin{equation*}
	\begin{aligned}
	v_- \cdot \sg^\text{u/l}_1 &= \pm \sg^\text{u/l}_1 \\
	v_- \cdot \sg^\text{u/l}_2 &= \mp \sg^\text{u/l}_2
	\end{aligned}
	\end{equation*}
	as if the marked cycle is red in \(\mathscr{S}_1\) then it is green in \(\mathscr{S}_2\) and vice versa.
	As the action descends to an action on \( \dkh' ( L ) \) we see that \( [\sg^\text{u/l}_1] \) is an eigenvector of the action of \( v_- \) of eigenvalue \( \pm 1 \) and \( [\sg^\text{u/l}_2] \) is an eigenvector of eigenvalue \( \mp 1 \), so that \( [\sg^\text{u/l}_1] \neq [\sg^\text{u/l}_2] \).
	
	At this point we have
	\begin{equation*}
	\text{rank} \left( \dkh' ( L ) \right) \geq 2 \left| \left\{ \text{alternately coloured smoothings of}~ L \right\} \right |.
	\end{equation*}
	In order to tighten this to an equality we shall again employ Gauss elimination along with the observation that the differential restricted to elements corresponding to non-alternately coloured smoothings is an isomorphism. In the case of the \( \Delta' \) and \( m' \) maps this is evident from \Cref{Eq:rgmaps}. Regarding the \( \eta' \) map, we have
	\begin{equation}
	\label{Eq:etaprime}
	\begin{aligned}
	\eta' ( r^\text{u}) &= r^\text{l} \qquad &\eta' ( r^\text{l}) &= 2r^\text{u}  \\
	\eta' ( g^\text{u}) &= g^\text{l} \qquad &\eta' ( g^\text{l}) &= -2g^\text{u} \\
	\end{aligned}
	\end{equation}
	so that \( \eta' \) is an isomorphism (we are working over \( \Q \)). Thus we Gauss eliminate elements associated to non-alternately coloured smoothings of \( L \) and arrive at the desired equality.
\end{proof}

We now return to \Cref{Eq:clee3}, in order to generalise it to the virtual case. It is clear that we have some work to do, as the virtual Hopf link (as depicted in \Cref{Fig:vH}), for example, has no alternately coloured smoothings (one readily sees that the generators on the right of \Cref{Fig:vH} will cancel in doubled Lee homology). Before describing the virtual situation we make some preliminary definitions.

\begin{definition}
	\label{Def:shadow} Let \( D \) be a virtual link diagram. Denote by \( S ( D ) \) the diagram obtained from \( D \) be removing the decoration at classical crossings; we refer to \( S ( D ) \) as the \emph{shadow} of \( D \). Let a \emph{component} of \( S ( D ) \) be an \( S^1 \) embedded in such a way that at a classical or virtual crossing we have exactly one of the following:
	\begin{itemize}
		\item All the incident arcs are contained in the component.
		\item The arcs contained in the component are not adjacent.
		\item None of the arcs are contained in the component.
	\end{itemize}
	Thus components of \( S ( D ) \) are in bijection with those of \( D \). \CloseDef
\end{definition}

\begin{definition}
	\label{Def:gauss}
	Let \( D \) be an \( n \)-component virtual link diagram and \( S ( D ) \) its shadow. Denote by \( G ( D ) \) the \emph{Gauss diagram of} \( D \), formed in the following manner:
	\begin{center}
		\begin{enumerate}[(i)]
			\item Place \( n \) copies of \( S^1 \) disjoint in the plane. A copy of \( S^1 \) is known as a \emph{circle} of \( G ( D ) \).
			\item Fix a bijection between the components of \( S ( D ) \) and the circles of \( G ( D ) \).
			\item Arbitrarily pick a basepoint on each component of \( S ( D ) \) and on the corresponding circle of \( G ( D ) \).
			\item Pick a component of \( S ( D ) \) and progress from the basepoint around that component (in either direction). When meeting a classical crossing label it and mark that label on the corresponding circle of \( G ( D ) \) (virtual crossings are ignored). Continue until the basepoint is returned to.
			\item Repeat for all components of \( S ( D ) \); if a crossing is met which already has a label, use it.
			\item Add a chord linking the two incidences of each label. These chords may intersect and have their endpoints on different circles of \( G ( D ) \).\CloseDef
		\end{enumerate}
	\end{center}
\end{definition}

Gauss diagrams are more commonly defined for diagrams, rather than shadows, of links but this definition contains all the information we require. An example of a shadow and of a Gauss diagram can be found in \Cref{Fig:sandg}.

\begin{definition}
	\label{Def:degen}
	A circle within a Gauss diagram is known as \emph{degenerate} if it contains an odd number of chord endpoints.\CloseDef
\end{definition}

\begin{figure}
	\includegraphics[scale=0.75]{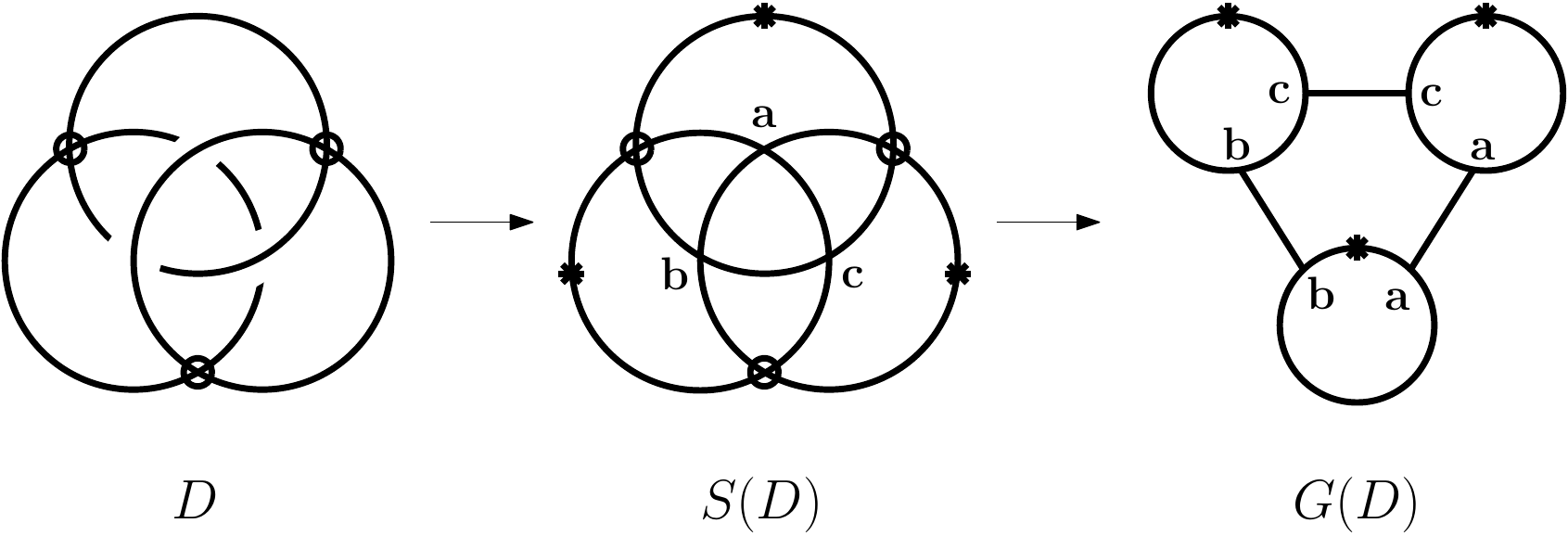}
	\caption{The shadow and Gauss diagram of a virtual link diagram.}
	\label{Fig:sandg}
\end{figure}

\begin{theorem}
	\label{Thm:acs}
	Given a diagram \( D \) of a virtual link \( L \)
	\begin{equation*}
		\begin{aligned}
			\left| \left\{ \text{alternately coloured smoothings of}~ L \right\} \right| &= \left| \left\{ \text{alternately coloured smoothings of}~ D \right\} \right| \\
			&= \left\lbrace \begin{matrix*}[l]
				2^{ | L | }, ~\text{if}~ G ( D ) ~\text{contains no degenerate circles} \\
				0, ~\text{otherwise.}
			\end{matrix*}
			\right.
			\end{aligned}
	\end{equation*}
\end{theorem}

\begin{proof}
	As stated above, the number of alternately coloured smoothings is a link invariant, so that we are free to use the Gauss diagram associated to any diagram of \( L \).
	
	As observed by Kauffman \cite{Kauffman2004b} alternately coloured smoothings of a link diagram are in bijection with particular colourings of the shadow of the diagram: colouring the arcs of the shadow either red or green such that at every classical crossing we have the following (up to rotation):
	\begin{center}
		\includegraphics[scale=1]{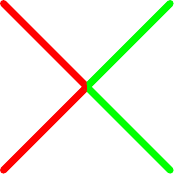}\label{Graphic:pc}
	\end{center}
	Such a colouring is known as a \emph{proper colouring}. Given a virtual link diagram \( D \) and a proper colouring of \( S ( D ) \), one produces an alternately coloured smoothing of \( D \) by resolving each classical crossing in the manner dictated by the proper colouring i.e.\ joining red to red and green to green. Two examples are given in \Cref{Fig:borroshadow}. It is easy to see that this association defines a bijection between the set of proper colourings and the set of alternately coloured smoothings.
	
	\begin{figure}
		\includegraphics[scale=0.75]{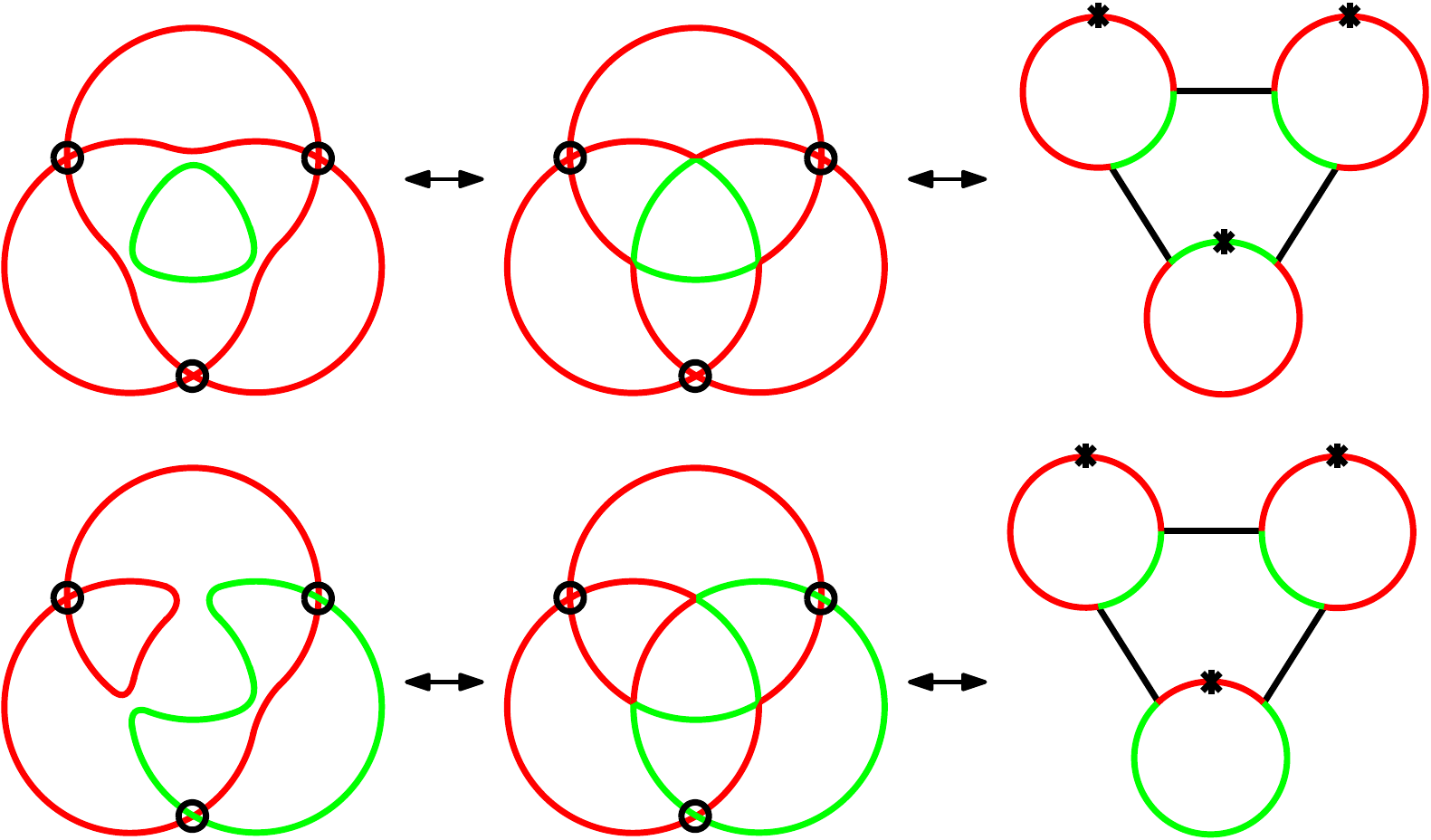}
		\caption{Alternately coloured smoothings (left) and Gauss diagrams (right) associated to proper colourings of a shadow (centre).}
		\label{Fig:borroshadow}
	\end{figure}

	Next, notice that a proper colouring of \( S ( D ) \) induces a colouring of the circles of \( G ( D ) \): colour the connected components of the complement of the chord endpoints in the manner dictated by the colouring of the shadow (so that when an endpoint is passed the colour changes). A Gauss diagram coloured in such a way is known as \emph{alternately coloured}. Examples are given in \Cref{Fig:borroshadow}. It is again easy to see that alternately coloured Gauss diagrams are in bijection with proper colourings, so that we have a bijection between alternately coloured smoothings of \( D \) and alternate colourings of \( G ( D ) \).
	
	In light of the above we see that we are required to verify that a Gauss diagram of \( n \) circles has \( 2^n \) alternate colourings if and only it has no degenerate circles.
	
	Let \( G ( D ) \) contain a degenerate circle. On this circle the number of connected components of the complement of the end points is odd, from which we deduce that it cannot be alternately coloured (as the colour must change when passing an endpoint). That there are \( 2^n \) alternate colourings if there is no degenerate circle follows from the observation that there are two possible configurations for each circle, and that given an alternate colouring flipping the configuration on one circle yields a new alternate colouring.
\end{proof}

\begin{corollary}
	Let \(K\) be a virtual knot. Then \( \text{rank} \left( \dkh' ( K ) \right) = 4 \) and \( \dkh' ( K ) \) is supported in homological degree equal to the height of the alternately colourable smoothing.
\end{corollary}
\begin{proof}
	Let \( D \) be a virtual knot diagram. Then \( G ( D ) \) satisfies the condition of \Cref{Thm:acs} as it contains only one circle, on which all chord endpoints must lie. Of course, every chord has two endpoints so that this circle must contain an even number of them. The statement then follows from \Cref{Thm:leerank}.
\end{proof}

Classically, the alternately colourable smoothing of an oriented knot diagram is its oriented smoothing. Classical Khovanov homology is rigged so that this smoothing is at height \( 0 \), and subsequently classical Lee homology of a knot is supported in homological degree \( 0 \). This is no longer the case with doubled Lee homology: virtual knot \( 2.1 \) (given in \Cref{Fig:21}) provides an example of a knot for which the alternately colourable smoothing is, in fact, the unoriented smoothing. Taking the connect sum of \( 2.1 \) with any classical knot yields a virtual knot for which the alternately colourable smoothing is neither the oriented nor the unoriented smoothing. The height of the alternately colourable smoothing of a knot shall be used in the definition of the doubled Rasmussen invariant in \Cref{Sec:Ras}, and is shown to be equal to the odd writhe of the knot in \Cref{Subsec:rasoddwrithe}.\label{Page:classicalacs}

\begin{corollary}
	\label{Cor:splitlink}
	Let \( L \) be a split link of \( n \) components. Then \( \text{rank} \left( \dkh' ( L ) \right) = 2^{n+1} \).
\end{corollary}

\subsection{Interaction with cobordisms}
\label{Subsec:cobordisms}

A cobordism between classical links defines a map on classical Lee homology; this behaviour is replicated by doubled Lee homology. Unlike the classical case, however, many connected cobordisms must be assigned the zero map, a consequence, for example, of the vanishing of \( \dkh ' ( L )  \) for certain links or of the possibility of doubled Lee homology of knots being supported in non-zero homological degrees. Nevertheless, there are classes of cobordisms for which the associated maps are non-zero (some of which we use in \Cref{Sec:applications}). In this section we verify that concordances and a certain class of arbitrary genus cobordisms are assigned non-zero maps.

We begin by stating some definitions regarding virtual cobordism (see \cite{Kauffman1998c}).

\begin{definition}
	\label{Def:cobordism}
	Two virtual links \( L_1 \) and \( L_2 \) are \emph{cobordant} if a diagram of one can be obtained from a diagram of the other by a finite sequence of births and deaths of circles, oriented saddles, and virtual Reidemeister moves. Such a sequence describes a compact, oriented surface, \( S \), such that \( \partial S = L_1 \sqcup L_2 \). Births of circles, oriented saddles, and deaths of circles correspond, respectively, to \(0\)-, \(1\)-, and \(2\)-handles of \(S\). If \( g ( S ) = 0 \) we say that \( L_1 \) and \( L_2 \) are \emph{concordant}. If \( K \) is a knot concordant to the unknot we say that \( K \) is \textit{slice}. In general, we define the \emph{slice genus} of a virtual knot \( K \), denoted \( g^\ast \), as
	\begin{equation*}
	g^\ast ( K ) = \min \lbrace g ( S ) ~|~ S ~\text{a compact, oriented, connected surface such that}~ \partial S = K \rbrace
	\end{equation*}
	(here we have simply capped off the unknot in \( \partial S \) with a disc).\CloseDef
\end{definition}

Virtual links are equivalence classes of embeddings of disjoint unions of \( S^1 \) into thickened surfaces \cite{Kuperberg2002}; thus the surface \( S \) is embedded in a \( 4 \)-manifold of the form \( M \times [ 0 , 1 ] \) where \( M \) is a compact, oriented \(3\)-manifold with \( \partial M = \Sigma_k \sqcup \Sigma_l \), where \( \Sigma_i \) denotes a closed surface of genus \( i \). The \( 3 \)-manifold \( M \) is described in the standard way in terms of codimension \( 1 \) submanifolds and critical points: starting from \( \partial M = \Sigma_k \), codimension \( 1 \) submanifolds are \( \Sigma_k \) until we pass a critical point, after which they are \( \Sigma_{k \pm 1} \). Critical points of \( M \) correspond to handle stabilisation (induced by virtual Reidemeister moves). A finite number of handle stabilisations are made to reach \( \Sigma_l \). We say that a link \emph{appears in} \(S\) if a diagram of it appears in the sequence of diagrams describing \(S\).

Of course, the subject of \Cref{Def:cobordism} is really a \emph{presentation} of a cobordism; two distinct presentations may be equivalent under isotopy (relative to their boundary), and represent the same cobordism. However, as mentioned by Rasmussen and others in the classical case, we expect isotopic presentations to be assigned the same map. Like Rasmussen, we do not pursue this further, and suffice ourselves with working with presentations. Moreover, we shall ignore the difference between a presentation and a cobordism, referring to the surface \( S \) as a \emph{cobordism between} \( L_1 \) and \( L_2 \).

\begin{definition}
	Let \( S \) be a cobordism between two virtual links \( L_1 \) and \( L_2 \) which is described by exactly one virtual Reidemeister move or one oriented \(0\)-,\(1\)-, or \(2\)-handle addition. Such a cobordism is known as \emph{elementary}.\CloseDef
\end{definition}

We wish to associate maps to cobordisms such that, where they are non-zero, the maps respect the filtration and send alternately coloured generators (of the homology of the intitial link) to linear combinations of alternately coloured generators (of the homology of the final link).

Of course, any cobordism can be built by gluing elementary cobordisms end to end, so we shall first investigate these simple cobordisms. In all there are ten of them: four given by the purely virtual Reidemeister moves and the mixed move, three given by the classical Reidemeister moves, and three given by the \(0\)-, \(1-\), and \(2\)-handle additions. We separate the work into elementary cobordisms which contain virtual Reidemeister moves and those which contain handle additions.

\noindent(\emph{Virtual Reidemeister moves}): Let \( D_1 \) and \( D_2 \) be diagrams of virtual links \( L_1 \) and \( L_2 \), and \( S \) an elementary cobordism between them which contains a purely virtual Reidemeister move or mixed move (as depicted in \Cref{Fig:vrm}). Then \( \cdkh ' ( D_1 ) = \cdkh ( D_2 ) \), as such moves preserve the number of cycles in a smoothing and the incoming and outgoing differentials. Thus we associate to \( S \) the map \( \phi_S = \text{id} : \dkh ' ( L_1 ) \rightarrow \dkh ' ( L_2 ) \). It is also clear that such a cobordism sends alternately coloured smoothings of \( D_1 \) to those of \( D_2 \), so that alternately coloured generators of \( \dkh ' ( L_1 ) \) are sent to those of \( \dkh ' ( L_1 ) \).

If \( S \) contains a classical Reidemeister move then \( \phi_S \) is one of the maps defined in \cite[Section \(6\)]{Rasmussen2010}, with the addition of the appropriate \( \phantom{}^{\text{u/l}} \) superscripts. We satisfy ourselves with a quick demonstration that classical Reidemeister moves send alternately coloured smoothings to alternately coloured smoothings, via proper colourings of shadows. As mentioned above, given a virtual link diagram \( D \), the set of its alternately coloured smoothings is in bijection with the set of proper colourings of its shadow. Let \( D \) and \( D' \) be related by a classical Reidemeister move. Then \( D \) and \( D' \) are identical except within a neighbourhood of the move. Given a proper colouring of \(  S ( D ) \) define a proper colouring of \( S ( D' ) \) which is identical to that of \( S ( D ) \) outside the proscribed neighbourhood; the colouring within is dictated by that of arcs incident to the neighbourhood. Some examples are given in \Cref{Fig:vrmcolours}. It is clear that this defines a bijection between the proper colourings of \( S ( D ) \) and those of \( S ( D' ) \), and it follows that the maps associated to the classical Reidemeister moves are isomorphisms on doubled Lee homology.

\begin{figure}
	\includegraphics[scale=0.65]{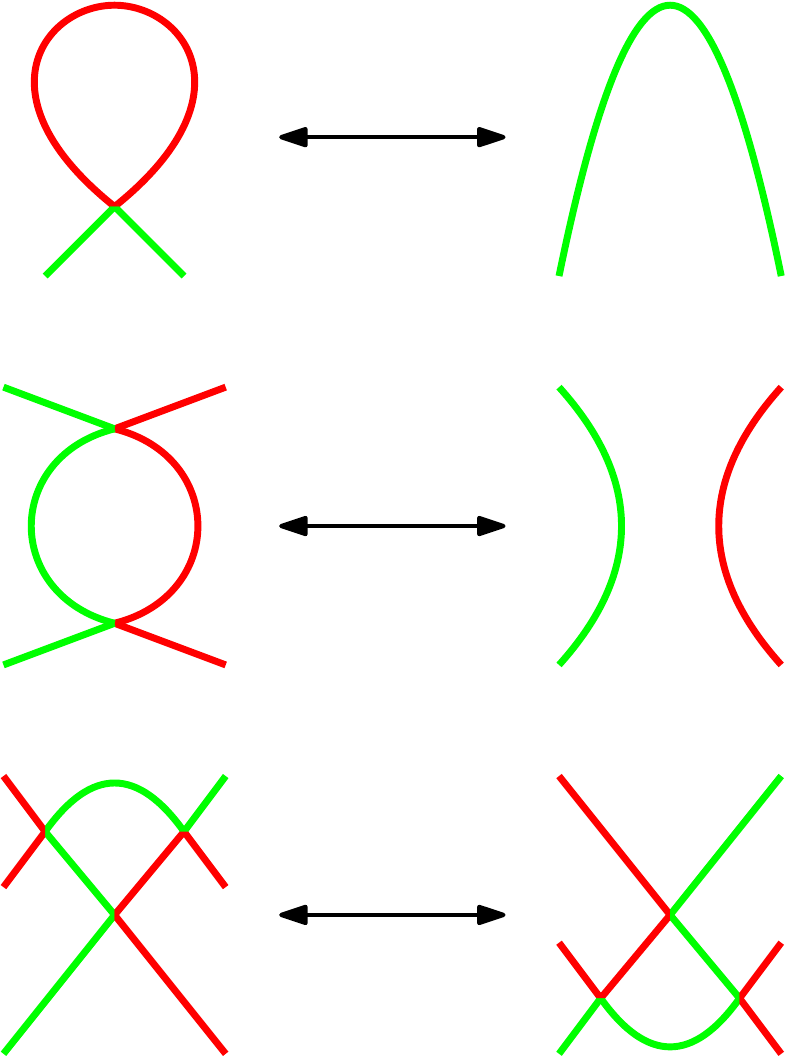}
	\caption{Examples of the affects of classical Reidemeister moves on proper colourings of shadows. Notice the endpoints of the arcs are coloured the same colour on the left- and right-hand sides.}
	\label{Fig:vrmcolours}
\end{figure}

\noindent(\emph{Handle additions})\label{Page:handles}: Let \( D_1 \) and \( D_2 \) be diagrams of virtual links \( L_1 \) and \( L_2 \), and \( S \) an elementary cobordism between them which contains a handle addition. Then \( S \) defines a map of cubes between the cube of smoothings of \( D_1 \) and that of \( D_2 \): removing a neighbourhood of the classical crossings of \( D_1 \) and \( D_2 \), both diagrams look identical except in the region in which the handle is attached. Moreover, as handle additions do not change the number of crossings of diagram, the smoothings of \( D_1 \) and \( D_2 \) are in bijection (a string of \(0\)'s and \(1\)'s defines uniquely a smoothing of \( D_1 \) and of \( D_2 \)). Let the map of cubes defined by \( S \) be the map which sends a smoothing of \( D_1 \) to the associated smoothing of \( D_2 \). As the diagrams are identical exept in a small region this map acts simply on smoothings, and depends on the handle addition contained in \( S \): 
\begin{itemize}
	\item \(0\)-handle: a cycle is added which shares no crossings with any other cycle or itself.
	\item \(1\)-handle: two cycles are merged into one cycle, one cycle is split into two, or one cycle is sent to one cycle (while the \(1\)-handle is necessarily oriented, it is nonetheless possible for it to induce such a term as a map of cubes.)
	\item \(2\)-handle: a cycle which shares no crossings with any other cycle or itself is removed.
\end{itemize}
Thus we define a map \( \psi : \cdkh ' ( D_1 ) \rightarrow \cdkh ' ( D_2) \), whose affect on the specific cycle or cycles involved is as follows (and acts as the identity on the uninvolved cycles)
\begin{itemize}
	\item \(0\)-handle: \( \iota ' : \mathbb{Q} \rightarrow \mathcal{A} \) where \( \iota ' ( 1 ) = \vulp \), so that \( \iota ( 1 ) \otimes \vup = (v_{++})^{\text{u}} \), for example.
	\item \(1\)-handle: either \( m' \), \( \Delta ' \), or \( \eta ' \) as dictated by the corresponding entry in map of cubes.
	\item \(2\)-handle: \( \epsilon ' : \mathcal{A} \rightarrow \mathbb{Q} \) where \( \epsilon ' ( \vulp ) = 0 \), \( \epsilon ' ( \vulm ) = 1 \).
\end{itemize}
We define \( \phi_S : \dkh ' ( L_1 ) \rightarrow \dkh ' ( L_2) \) to be the map induced by \( \psi \). Notice that \( \phi_S \) is filtered of degree \(1\) for \(0\)- and \(2\)-handle additions and filtered of degree \(-1\) for \(1\)-handle additions, and that it preserves homological degree.

\begin{definition}
	Let \( D_1 \) and \( D_2 \) be diagrams of virtual links \( L_1 \) and \( L_2 \), and \( S \) a cobordism between them. Then \( S \) can be decomposed as a finite union of elementary cobordisms, so that
	\begin{equation*}
	S = S_1 \cup S_2 \cup \cdots \cup S_n
	\end{equation*}
	where \( S_i \) is an elementary cobordism. Define \( \phi_S = \phi_{S_n} \circ \phi_{S_{n-1}} \circ \cdots \circ \phi_{S_1} \).\CloseDef
\end{definition}

It is possible that a map associated to a cobordism is necessarily zero, owing to the doubled Lee homology of a link (or links) appearing in it being trivial in particular degrees (or possibly every degree). Homological degrees which survive throughout a cobordism are important, therefore.

\begin{definition}
	\label{Def:shared}
	Let \( D_1 \) and \( D_2 \) be diagrams of virtual links \( L_1 \) and \( L_2 \), and \( S \) a cobordism between them such that the doubled Lee homology of every link appearing in it is non-trivial in homological degree \( k \). Such a homological degree is known as a \emph{shared} degree (of \( S \)).\CloseDef
\end{definition}

\begin{figure}
	\includegraphics[scale=0.65]{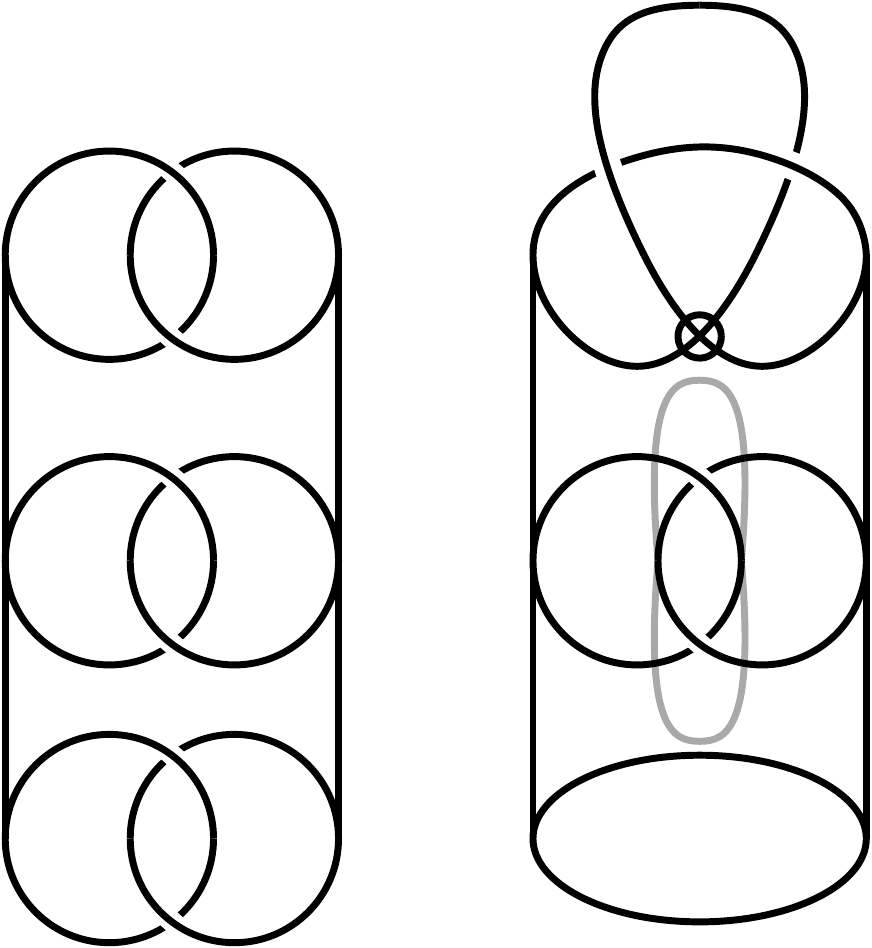}
	\caption{A cobordism with shared degree \(-2\) which is assigned the zero map. The grey lines denote that the right-hand component has genus: it is formed by a \(1\)-handle between a single component followed by a \(1\)-handle between two components.}
	\label{Fig:21tounknot}
\end{figure}

The existence of shared homological degrees is not enough to guarantee that a cobordism is assigned a non-zero map, however. Consider the cobordism depicted in \Cref{Fig:21tounknot}: the left-hand component is the identity cobordism on the classical Hopf link, while the right-hand component is a genus \(1\) cobordism from virtual knot \(2.1\) to the unknot. It can be quickly verified that \(-2\) is a shared degree of this cobordism, but that the map assigned to it is zero.

The remainder of this section is concerned with the task of verifying that concordances and some arbitrary genus cobordisms are assigned non-zero maps, as advertised above. In what follows, a cobordism is said to be \emph{weakly connected} if every connected component has a boundary component in the initial link.

\begin{theorem}
	\label{Thm:nonzero1}
	Let \( S \) be a concordance between a virtual knot \( K \) and a virtual link \( L \). Suppose that \( S \) contains no closed components and that \( \dkh ' ( L ) \neq 0 \). Then \( \phi_S \) is non-zero.
\end{theorem}

\begin{theorem}
	\label{Thm:nonzero2}
	Let \( S \) be a weakly connected cobordism between virtual knots \( K \) and \( K' \) with a non-empty set of shared degrees. Further assume that \( S \) can be decomposed as the union of two cobordisms, \( S_1 \) and \( S_2 \), such that \( S_1 \cap S_2 = L \) where \( L \) is a virtual link. Suppose that, in addition to virtual Reidemeister moves, \( S_1 \) contains only \(1\)-handles between a single link component, and \( S_2 \) handles between distinct link components. There is exactly one alternately colourable smoothing of \( L \), denoted \( \mathscr{S} \), such that the associated generators are in the image of \( \phi_{S_1} \), and a finite set of alternately colourable smoothings of \( L \), denoted \( \mathscr{L} \), such that the associated generators are in the cokernel of \( \phi_{S_2} \). Then \( \phi_{S} \) is non-zero if and only if \( \mathscr{S} \in \mathscr{L} \).
\end{theorem}
A cobordism satisfying the criteria of \Cref{Thm:nonzero2} is known as a \emph{targeted} cobordism.

We begin our path to the proofs of \Cref{Thm:nonzero1,Thm:nonzero2} by investigating elementary cobordisms; many maps assigned to them are non-zero automatically.

\begin{proposition}
	\label{Prop:nonzero1}
	Let \( D_1 \) and \( D_2 \) be diagrams of virtual links \( L_1 \) and \( L_2 \), and \( S \) an elementary cobordism between them which is a \(0\)- or \(2\)-handle addition, or a \(1\)-handle addition between two distinct link components. If \( L_1 \) has a non-zero number of alternately coloured smoothings, then \( S \) has shared degrees and \( \phi_S \) is non-zero in them.
\end{proposition}

\begin{proof}
	We are required to verify two criteria \( ( i ) \): that \( D_2 \) has at least one alternately coloured smoothing at the same height as one of the alternately coloured smoothings of \( D_1 \), and \( ( ii ) \): that \( \phi_S \) sends at least one alternately coloured generator of \( \dkh ' ( L_1 ) \) to a linear combination of those of \( \dkh ' ( L_2 ) \). For \(0\)- and \(2\)-handles \( ( i ) \) follows from the fact that the cycle being added or removed  does not take part in any of the crossings in \(D_1\) or \(D_2\), and thus places no restrictions on a smoothing being alternately coloured. As an handle addition does not change the number of classical crossings it is clear that an alternately coloured smoothing of \( D_1 \) is sent to an alternately coloured smoothing of \( D_2 \) of the same height. Further, noticing that
	\begin{equation}
	\label{Eq:02handlemaps}
	\begin{aligned}
	\iota ' ( 1 ) &= (r+g)^{\text{u/l}} \\
	\epsilon ' ( r^{\text{u/l}} ) &= -\dfrac{1}{2} \\
	\epsilon ' ( g^{\text{u/l}} ) &= \dfrac{1}{2}
	\end{aligned}
	\end{equation}
	we see that \( (ii) \) is satisfied. (Note that \(0\)-handles double the number of alternately coloured smoothings, while \(2\)-handles halve it.)
	
	For \(1\)-handle additions between two distinct link components we verify \( (i) \) in the following manner: consider the Gauss diagrams \( G ( D_1 ) \) and \( G ( D_2 ) \). By assumption \( G ( D_1 ) \) contains no degenerate circles. As the \(1\)-handle consituting \( S \) is between two distinct link components, \( G ( S ( D_2 ) ) \) can be obtained from \( G ( D_1 ) \) by combining two circles (those corresponding to the components between which the handle is added) and adding all chord endpoints which lie on them to the new circle, leaving the other circles unchanged. Thus the number of chord endpoints lying on the new circle must be a multiple of \(4\) and it is not degenerate. As the other circles are unchanged it is clear that \( G ( D_2 ) \) has no degenerate circles and \( D_2 \) has alternately coloured smoothings - note that it has half the number that \( D_1 \) has, however. That there are heights at which both \(D_1\) and \(D_2\) have alternately coloured smoothings again follows from the fact that handle additions do not change the number of classical crossings. The statement \( (ii) \) follows from \Cref{Eq:rgmaps,Eq:etaprime}. 
\end{proof}

In the case of \(1\)-handles involving a single link component we are able to determine whether they preserve the existence of alternately coloured smoothings by looking at their effect on Gauss diagrams. Using this we can specify the handle additions which are associated non-zero maps.

\begin{lemma}
	\label{Lem:acskiller}
	Let \( D_1 \) and \( D_2 \) be diagrams of virtual links \( L_1 \) and \( L_2 \), and \( S \) an elementary cobordism between them which is a \(1\)-handle addition involving a single link component. Further, assume \( \dkh ' ( L_1 ) \) is non-trivial. Then \( \dkh ' ( L_2 ) \) is trivial if and only if there is a proper colouring of \( S ( D_1 ) \) such that the handle addition is between two strands of opposite colour.
\end{lemma}

\begin{proof}
	It follows from \Cref{Thm:leerank} that \( \dkh ' ( L_2 ) \) is trivial if and only if \( D_2 \) has no alternately coloured smoothings. Consider the Gauss diagram of the shadow of \(D_1\): as the handle addition is between a single link component it can be represented in the following manner:
	\begin{center}
		\includegraphics[scale=0.75]{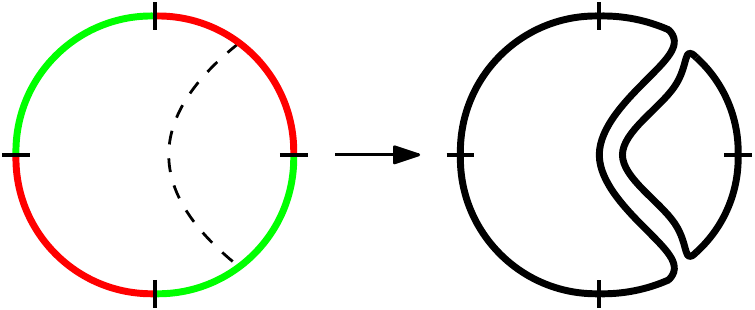}
	\end{center}
	On the left the circle of \( G ( D_1 ) \) corresponding to the component of \(D_1\) undergoing the handle addition is depicted; the dotted line shows the location of the handle addition. Clearly, if the handle is added between two regions of opposite colour the dotted line must enclose an odd number of chord endpoints, so that the newly created circles are degenerate (as depicted on the right). Conversely, it is easy to see that if the handle is between two regions of the same colour then the newly created circles are non-degenerate. To conclude, note the regions are either both coloured the same colour in all proper colourings of \( S (D_1) \) or are coloured opposite colours in all proper colourings, as all proper colourings are related by flipping the colours on a finite number of circles.
\end{proof}

\begin{corollary}
	\label{Cor:nonzero2}
	Let \( D_1 \) and \( D_2 \) be diagrams of virtual links \( L_1 \) and \( L_2 \), and \( S \) an elementary cobordism between them which contains a \(1\)-handle addition between a single link component. Further, assume that \( \dkh ' ( L_1 ) \) is non-trivial and that the \(1\)-handle addition is between strands of the same colour in \( S ( D_1 ) \). Then \( S \) has shared degrees and \( \phi_S \) is non-zero in them.
\end{corollary}

We omit the proof of \Cref{Cor:nonzero2} as it uses very similar ideas to that of \Cref{Prop:nonzero1} along with \Cref{Eq:rgmaps,Eq:etaprime}.

Using the map of cubes defined by a handle addition (see \cpageref{Page:handles}) we continue to investigate the maps associated to \(1\)-handle additions further. In what follows we shall suppress the upper/lower subscripts of the generators \( \sg \), as it easy to see that \( \sg^{\text{u}} \in \text{im} ( \phi_{S} ) \) if and only if \( \sg^{\text{l}} \in \text{im} ( \phi_{S} ) \). Also, whenever we state equalities such as \( \phi_S ( \sg ) = \sg ' \), for example, we shall always mean equality up to a (non-zero) scalar.

\begin{proposition}
	\label{Prop:1handles1}
	Let \( D_1 \) and \( D_2 \) be diagrams of virtual links \( L_1 \) and \( L_2 \), and \( S \) an elementary cobordism between them which is a \(1\)-handle addition. Further let \( \dkh ' ( L_1 ) \) and \( \dkh ' ( L_2 ) \) be non-trivial. (Recall that the smoothings of \(D_1\) and \(D_2\) are in bijection.) There are two cases:
	\begin{enumerate}[(i)]
		\item if the \(1\)-handle is between two distinct components of \( L_1 \), then every alternately coloured smoothing of \( D_2 \) is associated to an alternately coloured smoothing of \( D_1 \).
		\item if the \(1\)-handle involves a single component of \(L_1\), then every alternately coloured smoothing of \(D_1\) is associated to an alternately coloured smoothing of \(D_2\).
	\end{enumerate}
	(A smoothing of \(D_1\) is associated to a smoothing of \(D_2\) if and only if it is sent to it under the map of cubes defined by \(S\).)
\end{proposition}

\begin{proof}
	As observed in \Cref{Sec:leehom} the alternately coloured smoothings of a diagram are in bijection with the proper colourings of the shadow of the diagram. In case \( (ii) \) one readily observes that a proper colouring of \( S ( D_1 ) \) defines a proper colouring of \( S ( D_2 ) \) (as the handle must join two strands of the same colour, a consequence of \Cref{Lem:acskiller}). Moreover this proper colouring of \(S ( D_2 ) \) induces the same crossing resolutions as those of the proper colouring of \( S ( D_1 ) \), so that corresponding alternately coloured smoothings are associated. In case \( (i) \), notice that the reverse cobordism (from \(L_2\) to \(L_1\)) satisfies \( (ii) \).
\end{proof}

\begin{corollary}
	\label{Cor:surjection}
	Let \( D_1 \) and \( D_2 \) be diagrams of virtual links \( L_1 \) and \( L_2 \), and \( S \) an elementary cobordism between them which is a \(1\)-handle addition with shared degrees. Then, for \(k\) a shared degree
	\begin{enumerate}[(i)]
		\item If the handle addition is between two distinct components of \( L_1 \) then \( \phi_S \) surjects onto \( \bigoplus_i {\dkh '}_{k} ( L_2 ) \).
		\item If the handle addition is between a single component of \( L_1 \) then for all \( \sg \in \bigoplus_i {\dkh '}_{k} ( L_1 ) \) \( \phi_S ( \sg ) \neq 0 \).
	\end{enumerate}
\end{corollary}

\begin{proof}
	\( ( i ) \): Let \( \sg_2 \in {\dkh '}_{k} ( L_2 ) \) be defined by an alternately coloured smoothing \( \mathscr{S}_2 \) of \( D_2 \). Then by \Cref{Prop:1handles1} \( \mathscr{S}_2 \) is associated to \( \mathscr{S}_1 \), an alternately coloured smoothing of \( D_1 \) (and is mapped to it under the map of cubes defined by \( S \)). Let \( \sg_1 \) denote the alternately coloured generator of \( {\dkh '}_{k} ( L_1 ) \) defined by \( \mathscr{S}_1 \). If \( \phi_S \) acts by either \( \Delta ' \) or \( \eta ' \) on \( \sg_1 \) then \( \phi_S ( \sg_1 ) = \sg_2 \) automatically (by \Cref{Eq:rgmaps,Eq:etaprime}). If it acts by \( m ' \), then it is possible that \( \phi_S ( \sg_1 ) = 0 \), if the cycles undergoing the merge map are coloured opposite colours. Notice that if \( \mathscr{S}_2 \) is obtained from \( \mathscr{S}_1 \) by merging two cycles, then \( \mathscr{S}_1 \) is obtained from \( \mathscr{S}_2 \) by splitting two cycles. As observed in the proof of \Cref{Prop:1handles1}, by looking at proper colourings \( S ( D_2 ) \) and \( S ( D_1 ) \) associated to \( \mathscr{S}_2 \) and \( \mathscr{S}_1 \), respectively, we see that the relevant cycles cannot be coloured opposite colours in \( \mathscr{S}_1 \); thus \( \phi_S ( \sg_1 ) = \sg_2 \) (again by \Cref{Eq:rgmaps}).
	
	\noindent\( (ii) \): Let \( \sg \in \bigoplus_i {\dkh '}_{k} ( L_1 ) \) be defined by the alternately coloured smoothing \( \mathscr{S} \) of \( D_1 \). By \Cref{Lem:acskiller} the handle addition must be between cycles of the same colour in \( \mathscr{S} \) so that \( \phi_S ( \sg ) \neq 0 \) by \Cref{Eq:rgmaps,Eq:etaprime}.
\end{proof}

\begin{proof}[Proof of \Cref{Thm:nonzero1}]
	First we shall prove a fact about links appearing in concordances, before using this fact and an induction argument to prove the theorem in this restricted case.
	
	\begin{figure}
		\includegraphics[scale=0.65]{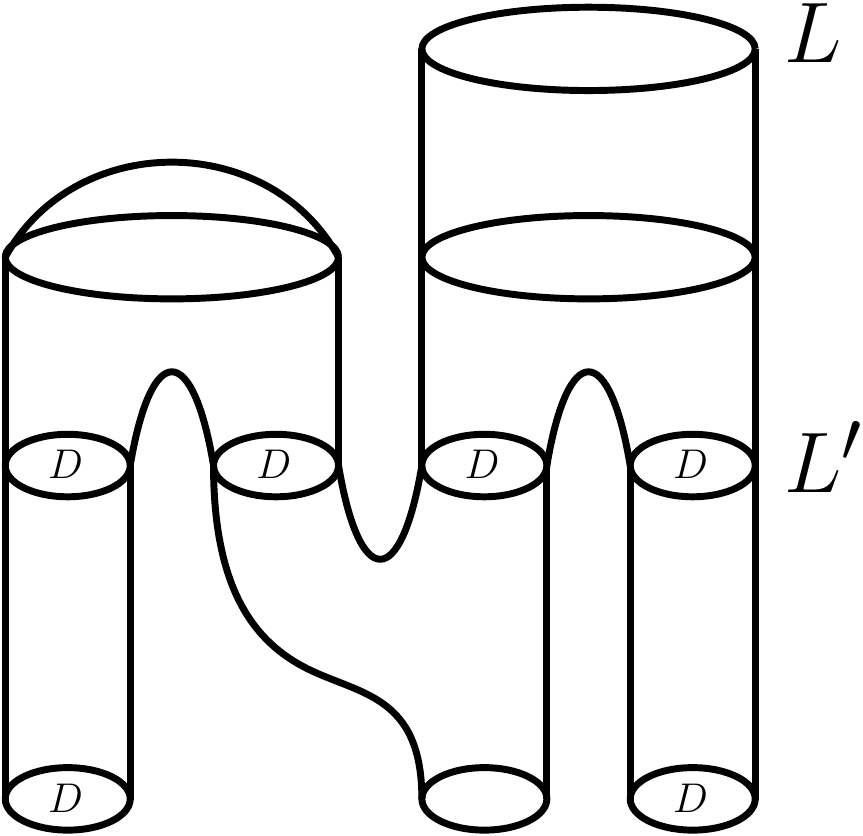}
		\caption{Cancelling degenerate components. The label \( D \) denotes a degenerate component.}
		\label{Fig:degeneratecobordism}
	\end{figure}
	
	Let \( S \) be a concordance between a virtual knot \( K \) and a virtual link \( L \) such that \( \dkh ' ( L ) \neq 0 \). Assume towards a contradiction that a link, \( \widetilde{L} \), appearing in \( S \) is such that \( \dkh ' ( \widetilde{L} ) = 0 \). By \Cref{Thm:leerank,Thm:acs}, \( S ( D ) \) must contain a degenerate circle, for \( D \) any diagram of \( \widetilde{L} \). Further, by \Cref{Lem:acskiller}, we see that degenerate circles are always created in pairs in a cobordism, and that degenerate circles can be cancelled against one another to produce non-degenerate circles (see \Cref{Fig:degeneratecobordism}). This cancelling process is as follows: add a \(1\)-handle between the components of \( \widetilde{L} \) which correspond to the degenerate circles, producing a new circle. Let the two initial degenerate circles be \( C_1 \) and \( C_2 \), and \( N_i \) denote the number of chord endpoints lying on \( C_i \). It is easy to see that the number of chord endpoints lying on the newly created circle is \( N = N_1 + N_2 \), and that \( N \) must be even as \( N_1 \) and \( N_2 \) are odd. Thus the newly created circle is non-degenerate.
	
	In what follows we shall call a component of a link diagram \emph{degenerate} if the circle corresponding to it in the associated Gauss diagram is degenerate. We may also speak of \emph{degenerate} components of links, as virtual Reidemeister moves cannot change the mod \(2\) number of chord endpoints lying on a circle.
	
	As \( K \) has non-trivial doubled Lee homology (it is a knot), no diagram of it contains a degenerate component. Therefore at least one \(1\)-handle involving a single link component must occur in \( S \) to produce \( \widetilde{L} \) (recall again \Cref{Lem:acskiller}). As \( L \) also has non-trivial doubled Lee homology, we see that we must remove all degenerate link components (by the process outlined above) in order to reach \( L \) from \( \widetilde{L} \). But degenerate circles are always formed in pairs, and we see that an attempt to cancel them all without introducing genus - which we are prohibited from doing as \( S \) is a concordance - leads to a non-compact situation; consider \Cref{Fig:degeneratecobordism}. As we are considering only compact cobordisms (recall \Cref{Def:cobordism}) we arrive at the desired contradiction.
	
	We now present the aforementioned induction argument: we shall build up concordances with elementary cobordisms. Let \( S' \) be a concordance between a virtual knot \( J \) and virtual link \( L_1 \) (distinct from \( K \), \( L \), and \( \widetilde{L} \) above) such that \( S' \) contains no closed components, \( \dkh ' ( L_1 ) \neq 0 \) and \( \phi_S \) is non-zero. We claim that if \( S_e \) is an elementary cobordism between \( L_1 \) and \( L_2 \) such that \( g ( S' \cup S_e ) = 0 \) then \( \phi_{S' \cup S_e} \) is non-zero also. Note that the argument above implies that we may restrict to the case in which \( \dkh ' ( L_2 ) \neq 0 \): if this did not hold then \( S' \cup S_e \) could not form part of a concordance between links which both have non-trivial doubled Lee homology.
	
	If \( S_e \) is a virtual Reidemeister move or a \(0\)-handle addition then \( \phi_{S' \cup S_e} \) is non-zero as \( \phi_{S_e} \) has trivial kernel. If \( S_e \) is a \(2\)-handle addition then \( \text{ker} ( \phi_{S_e} ) \) is spanned by the image of the map associated to a \(0\)-handle addition. But if a \(0\)-handle addition preceeds \( S_e \) then \( S' \cup S_e \) would contain a closed component, which it does not by assumption, so that \( \phi_{S' \cup S_e} \) is non-zero.
	
	If \( S_e \) is a \(1\)-handle involving a single link component we see that \( \phi_{S_e} \) has trivial kernel by \Cref{Cor:surjection}, as we are working in the case in which \( \dkh ' ( L_2 ) \neq 0 \).
	
	We are left with the case in which \( S_e \) is a \(1\)-handle between distinct link components. If \( S' \cup S_e \) is to have genus \(0\) the link components of \( L_1 \) involved in \( S_e \) must be belong to different connected components of \( S' \). As \( S' \) begins with \( J \), a virtual knot, at least one of the components of \( S' \) involved in \( S_e \) must be have no boundary component in \( J \) i.e.\ its first appearance in \( S' \) is a \(0\)-handle. (Cutting the cobordism depicted in \Cref{Fig:degeneratecobordism} at the link labelled \( L' \) yields an example.) 
	
	Let \( x \in \text{im} ( \phi_{S'} ) \). We can write \( x = \sum_i \sg_i \), where \( \sg_i \) is an alternately coloured generator of \( L_1 \). Let \( \mathscr{S}_i \) denote the alternately coloured smoothing of \( L_1 \) which defines \( \sg_i \), and \( \mathscr{C}_i \) the associated proper colouring of the shadow (of the appropriate diagram) of \( L_1 \). Then \( \phi_{S' \cup S_e} ( x ) = 0 \) if and only if the link components of \( L_1 \) involved in \( S_e \) are coloured opposite colours in every \( \mathscr{C}_i \) (recall the bijection between components of a link diagram and components of its shadow given in \Cref{Def:shadow}). This can be seen from \Cref{Eq:rgmaps}.
	
	As observed above, at least one of the connected components of \( S' \) involved in \( S_e \) begins with a \(0\)-handle, and \Cref{Eq:02handlemaps} shows that the image of the map assigned to a \(0\)-handle is a linear combination of both red and green. Therefore, given an arc of \( S ( L_1 ) \) lying on a component which begins with a \(0\)-handle, if \( \mathscr{C}_i \) has the arc labelled a particular colour, there must exist a \( \mathscr{C}_j \) in which the arc is coloured the opposite colour, and \( \phi_{S' \cup S_e} \) is non-zero.
	
	The base cases of the induction are the elementary cobordisms: they are all clearly of genus \(0\) and satisfy the induction hypothesis, under our assumption that both the initial and terminal links have non-trivial doubled Lee homology. Thus, given a concordance between a virtual knot and a virtual link with non-trivial doubled Lee homology, the assigned map is non-zero.
\end{proof}

\begin{proof}[Proof of \Cref{Thm:nonzero2}]
	Let \( S \), \( \mathscr{S} \), and \( \mathscr{L} \) be as in the theorem statement. The cobordism \( S_2 \) is a concordance between a link and a knot, and so \( \phi_{S_2} \) is non-zero by \Cref{Thm:nonzero1}. Recall that \( \mathscr{L} \) is the set of alternately colourable smoothings of \( L \) such that the associated generators are in the cokernel of \( \phi_{S_2} \); this is necessarily non-empty as \( \phi_{S_2} \) is non-zero. Further, \( \phi_{S_1} \) has trivial kernel (as observed in the proof of \Cref{Thm:nonzero1}) so that \( \text{img} ( \phi_{S_1} ) \) is of rank \(4\), spanned by the alternately coloured generators associated to exactly one smoothing of \( L \), namely \( \mathscr{S} \): to see this, recall that \( \phi_{S_1} (\sg^{\text{u}}) \neq 0 \iff \phi_{S_1} (\sg^{\text{l}}) \neq 0 \iff \phi_{S_1} (\overline{\sg}^{\text{u}}) \neq 0 \iff \phi_{S_1} (\overline{\sg}^{\text{l}}) \neq 0 \), so that if the generators associated to more than one smoothing of \( L \) lay in the image of \( \phi_{S_1} \), it could not be of rank \(4\) (generators associated to different smoothings are linearly independent). By assumption \( \mathscr{S} \in \mathscr{L} \) so that, if \( \sg^{\text{u/l}}, \overline{\sg}^{\text{u/l}} \) are the generators associated to \( \mathscr{S} \), then \( \sg^{\text{u/l}}, \overline{\sg}^{\text{u/l}} \in \text{coker} ( \phi_{S_2} ) \) and \( \phi_{S} = \phi_{S_2} \circ \phi_{S_1} \) is non-zero.
\end{proof}

\begin{remark}
	In proving \Cref{Thm:nonzero1,Thm:nonzero2} we could not follow Rasmussen's approach of propagating orientations through the cobordism, as we no longer necessarily have the relationship between orientations of a link and its alternately coloured smoothings. Also, while all the maps associated to elementary cobordisms are non-zero (as long as the homologies do not vanish), the full map associated to \(S\) may fail to be non-zero without requiring a non-empty set of shared degrees (in the classical case every cobordism has shared degree \(0\)). Moreover, the proof in the classical case is concerned only with this degree, while we must investigate the map associated to cobordisms in every homological degree.
\end{remark}

\section{A doubled Rasmussen invariant}
\label{Sec:Ras}
As demonstrated in the preceding section, for an oriented virtual knot, \( K \), \( \dkh ' ( K ) \) is a rank \( 4 \) bigraded group, supported in a single homological degree which can be determined easily from any diagram of \( K \). In \Cref{Subsec:rasdefinition} we show that the data provided by the quantum gradings in which \( \dkh ' ( K ) \) is supported are equivalent to a single integer (in the classical case this integer is necessarily even), so that the information contained in \( \dkh ' ( K ) \) is equivalent to a pair of integers. In \Cref{Subsec:rasproperties} we give some properties of this pair of integers. and in \Cref{Subsec:rasoddwrithe} we show that one of the members of the pair is equal to the odd writhe of the given knot. Finally, in \Cref{Subsec:leftmost} we describe a class of knots for which the invariant can be quickly calculated.

\subsection{Definition}\label{Subsec:rasdefinition}
We referred to a filtration of \( \cdkh ' ( K ) \) in \Cref{Def:dkhprime} - let us concretise it (following Rasmussen \cite{Rasmussen2010}). Let \( D \) be an oriented virtual knot diagram of \( K \) with \( n_+ \) positive classical crossings and \( n_- \) classical crossings. The homological grading on \( \cdkh ' ( K ) \), denoted \( i \), is as defined in \Cref{Eq:doubled}. The quantum grading is the standard one: define \( p ( \vup ) = 1 \), \( p ( \vum ) = -1 \), \( p ( \vlp ) = 0 \), \( p ( \vlm ) = -2 \), \( p ( \bigotimes x ) = \sum p ( x ) \), then the quantum grading is the shift \( j ( x ) = p ( x ) + i ( x ) + n_+ - n_- \). Let \( \mathcal{F}_k = \lbrace x \in \cdkh ' ( K ) ~|~ j ( x ) \geq k \rbrace \), so that we have the filtration
\begin{equation*}
	0 = \mathcal{F}_n \subset \mathcal{F}_{n-1} \subset \cdots \subset \mathcal{F}_m = \cdkh ' ( K )
\end{equation*}
for some \( n,m \in \mathbb{Z} \); let \( s \) denote the associated grading i.e.\ \( s ( x ) = k \) if \( x \in \mathcal{F}_k \) and \( x \notin \mathcal{F}_{k+1} \).

\begin{definition}
	\label{Def:sminmax} For a virtual knot \( K \) let
	\begin{equation}
		\begin{aligned}
			s^{\text{u}}_{\text{max}} ( K ) &= \text{max} \lbrace s ( x ) ~|~ x \in \dkh ' ( K ), ~ x \neq 0, ~ x \in \mathcal{A}^{\otimes n} \rbrace \\
			s^{\text{l}}_{\text{max}} ( K ) &= \text{max} ( K ) \lbrace s ( x ) ~|~ x \in \dkh ' ( K ), ~ x \neq 0, ~ x \in \mathcal{A}^{\otimes n} \lbrace -1 \rbrace \rbrace
		\end{aligned}
	\end{equation}
	and similarly define \( s^{\text{u/l}}_{\text{min}} ( K ) \).\CloseDef
\end{definition}

That \( s^{\text{u/l}}_{\text{max}} ( K ) \) can be determined from \( s^{\text{u/l}}_{\text{min}} ( K )  \) (and vice versa) follows in large part from the following augmented version of \cite[Lemma \(3.5\)]{Rasmussen2010}.

\begin{lemma}
	\label{Lem:evenodd}
	For a virtual knot \( K \)
	\begin{equation*}
		\dkh ' ( K ) = {\dkh ' }_1 ( K ) \oplus {\dkh ' }_2 ( K ) \oplus {\dkh ' }_3 ( K ) \oplus {\dkh ' }_0 ( K )
	\end{equation*}
	where \( {\dkh ' }_i ( K ) \) is generated by elements of quantum grading congruent to \( i \mod 4 \). Further
	\begin{enumerate}[(i)]
		\item Either
		\begin{equation*}
			\begin{aligned}
				&\mathfrak{s}^{\text{u}} \pm \overline{\mathfrak{s}}^{\text{u}} \in {\dkh ' }_1 ( K )  \\
				&\mathfrak{s}^{\text{u}} \mp \overline{\mathfrak{s}}^{\text{u}} \in {\dkh ' }_3 ( K )
			\end{aligned}
		\end{equation*}
		or
		\begin{equation*}
		\begin{aligned}
			&\mathfrak{s}^{\text{u}} \pm \overline{\mathfrak{s}}^{\text{u}} \in {\dkh ' }_0 ( K ) \\
			&\mathfrak{s}^{\text{u}} \mp \overline{\mathfrak{s}}^{\text{u}} \in {\dkh ' }_2 ( K ).
		\end{aligned}
		\end{equation*}
		\item Either
		\begin{equation*}
		\begin{aligned}
			&\mathfrak{s}^{\text{u}} \pm \overline{\mathfrak{s}}^{\text{u}} \in {\dkh ' }_{1/3} ( K ) \\
			&\mathfrak{s}^{\text{l}} \pm \overline{\mathfrak{s}}^{\text{l}} \in {\dkh ' }_{0/2} ( K )
		\end{aligned}
		\end{equation*}
		or
		\begin{equation*}
		\begin{aligned}
			&\mathfrak{s}^{\text{u}} \pm \overline{\mathfrak{s}}^{\text{u}} \in {\dkh ' }_{0/2} ( K ) \\
			&\mathfrak{s}^{\text{l}} \pm \overline{\mathfrak{s}}^{\text{l}} \in {\dkh ' }_{3/1} ( K ).
		\end{aligned}
		\end{equation*}
	\end{enumerate}
Here \( \sg^{\text{u/l}} \) denotes an alternately coloured generator as defined in \Cref{Eq:acgenerators}, and \( \overline{\sg}^{\text{u/l}} \) denotes the generator formed by replacing \( r \) with \( g \) and \( g \) with \( r \).
\end{lemma}

\begin{proof}
	That \( \dkh ' ( K ) \) decomposes into the given direct sum follows from the form of the differential: a part graded of degree \(0\) and other graded of degree \( 4 \). The statements within \( ( ii ) \) are obvious consequences of the construction of \( \sg^{\text{u/l}} \).
	
	We are left with \( (i) \): the mod \( 4 \) behaviour of the quantum grading is complicated by the fact that doubled Khovanov homology is supported in both odd and even quantum gradings, a departure from the classical case. We shall prove the case when \( s ( \sg^{\text{l}} ) \in 2 \mathbb{Z} \); this corresponds to the first statement in \( (i) \), the second follows identically modulo a grading shift.
	
	Following Rasmussen, define \( \iota : \dkh ' ( K ) \rightarrow \dkh ' ( K ) \) so that \( \iota \) acts by the identity on \( {\dkh ' }_0 ( K ) \oplus {\dkh ' }_1 ( K ) \) and by multiplication by \( -1 \) on \( {\dkh ' }_2 ( K ) \oplus {\dkh ' }_3 ( K ) \). Next, define \( \mathfrak{i} : \mathcal{A} \rightarrow \mathcal{A} \) by \( \mathfrak{i} ( v_+ ) = v_+ \) and \( \mathfrak{i} ( v_- ) = -v_- \). Then \( \mathfrak{i} ( r ) = g \) and \( \mathfrak{i} ( g ) = r \), and \( \mathfrak{i}^{\otimes n} : \mathcal{A}\lbrace -1 \rbrace \rightarrow \mathcal{A}\lbrace -1 \rbrace \) acts as the identity on \( {\dkh ' }_0 ( K ) \) and by multiplication by \( -1 \) on \( {\dkh ' }_2 ( K ) \). Thus we have
	\begin{equation*}
		\iota ( \sg^{\text{l}} ) = \mathfrak{i} ( \sg^{\text{l}} ) = \overline{\sg}^{\text{l}}
	\end{equation*}
	which yields
	\begin{equation*}
		\begin{aligned}
			\iota ( \sg^{\text{l}} + \overline{\sg}^{\text{l}} ) = \mathfrak{i} ( \sg^{\text{l}} + \overline{\sg}^{\text{l}} ) &= \sg^{\text{l}} + \overline{\sg}^{\text{l}} \\
			\iota ( \sg^{\text{l}} - \overline{\sg}^{\text{l}} ) = \mathfrak{i} ( \sg^{\text{l}} - \overline{\sg}^{\text{l}} ) &= - \left( \sg^{\text{l}} - \overline{\sg}^{\text{l}} \right)
		\end{aligned}
	\end{equation*}
	from which we deduce that \( \sg^{\text{l}} + \overline{\sg}^{\text{l}} \in {\dkh ' }_0 ( K ) \) and \( \sg^{\text{l}} - \overline{\sg}^{\text{l}} \in {\dkh ' }_2 ( K ) \). We conclude by invoking \( ( ii ) \).
\end{proof}

\begin{corollary}
	\label{Cor:maxmin}
	Let \(K\) be a virtual knot. Then
	\begin{equation*}
		s^{\text{u/l}}_{\text{max}} ( K ) > s^{\text{u/l}}_{\text{min}} ( K ).
	\end{equation*}
\end{corollary}

\begin{proposition}
	\label{Prop:minmaxrelation}
	Let \(K\) be a virtual knot. Then
	\begin{equation*}
		s^{\text{u/l}}_{\text{max}} ( K ) = s^{\text{u/l}}_{\text{min}} ( K ) + 2.
	\end{equation*}
\end{proposition}

\begin{proof}
	Consider the map
	\begin{equation*}
		\partial : \dkh ' ( K \sqcup \raisebox{-1.75pt}{\includegraphics[scale=0.3]{unknotflat.pdf}} ) \rightarrow \dkh ' ( K )
	\end{equation*}
	induced by the connect sum \( K \# \raisebox{-1.75pt}{\includegraphics[scale=0.3]{unknotflat.pdf}} = K \) (this is well-defined as it is between \(K\) and a crossingless unknot diagram). This is well-defined, preserves homological degree, and with respect to the quantum degree is graded of degree \( -1 \) (as it is simply \( \text{id} \otimes m ' \)). Again we follow Rasmussen and denote the alternately coloured generators of \( \dkh ' ( K ) \) by their decoration at the connect sum site i.e.\ \( \sg^{\text{u/l}}_{r} \) and \( \sg^{\text{u/l}}_{g} \). The alternately coloured generators of \( \dkh ' ( K \sqcup \raisebox{-1.75pt}{\includegraphics[scale=0.3]{unknotflat.pdf}} ) \) are then \( \sg^{\text{u/l}}_{r} \otimes r^{\text{u/l}} \), \( \sg^{\text{u/l}}_{r} \otimes g^{\text{u/l}} \), \( \sg^{\text{u/l}}_{g} \otimes r^{\text{u/l}} \), and \( \sg^{\text{u/l}}_{g} \otimes g^{\text{u/l}} \). Under \( \partial \) we have
	\begin{equation*}
		\begin{aligned}
			\partial ( \sg^{\text{u/l}}_{r} \otimes g^{\text{u/l}} ) &= \partial ( \sg^{\text{u/l}}_{g} \otimes r^{\text{u/l}} ) = 0 \\
			\partial ( \sg^{\text{u/l}}_{r} \otimes r^{\text{u/l}} ) &= \sg^{\text{u/l}}_{r} \\
			\partial ( \sg^{\text{u/l}}_{g} \otimes g^{\text{u/l}} ) &= \sg^{\text{u/l}}_{g}.
		\end{aligned}
	\end{equation*}
	Noticing that \( s^{\text{u/l}}_{\text{max}} ( K ) = s ( \sg^{\text{u/l}}_{r} \pm \sg^{\text{u/l}}_{g} ) \) and
	\begin{equation*}
		\partial ( ( \sg^{\text{u/l}}_{r} \pm \sg^{\text{u/l}}_{g} ) \otimes r^{\text{u/l}} ) = \sg^{\text{u/l}}_{r}
	\end{equation*}
	we obtain
	\begin{equation*}
		\begin{aligned}
			s ( ( \sg^{\text{u/l}}_{r} \pm \sg^{\text{u/l}}_{g} ) \otimes r^{\text{u/l}} ) &\leq s ( \sg^{\text{u/l}}_{r} ) + 1 \\
			s^{\text{u/l}}_{\text{max}} ( K ) -1 &\leq s^{\text{u/l}}_{\text{min}} ( K ) + 1
		\end{aligned}
	\end{equation*}
	as \( \partial \) is graded of degree \( -1 \) (that \( s^{\text{u/l}}_{\text{min}} ( K ) =  s ( \sg^{\text{u/l}}_{r} ) \) follows from \Cref{Lem:evenodd}).
\end{proof}

Thus any of the four quantities defined in \Cref{Def:sminmax} determines all of the others and we able to make the following definition.

\begin{definition}
	\label{Def:rasinvariant}
	For a virtual knot \( K \) let \( \mathbbm{s} ( K ) = ( s_1 ( K ), s_2 ( K ) ) \in \mathbb{Z} \times \mathbb{Z} \) where
	\begin{equation*}
		\begin{aligned}
			s_1 ( K ) &= s^{\text{l}}_{\text{max}} ( K ) \\
			s_2 ( K ) &= i ( \sg^{\text{u/l}} ) = | \mathscr{S} |
		\end{aligned}
	\end{equation*}
	where \( i \) denotes homological grading and \( \sg^{\text{u/l}} \) an alternately coloured generator of \( K \) associated to the alternately coloured smoothing \( \mathscr{S} \). We refer to \( \mathbbm{s} ( K ) \) as the \emph{doubled Rasmussen invariant} of \( K \).\CloseDef
\end{definition}

\subsection{Properties}\label{Subsec:rasproperties}
\begin{proposition}
	For a classical knot \( K \) \( \mathbbm{s} ( K ) = ( s ( K ), 0 ) \), where \( s ( K ) \) denotes the classical Rasmussen invariant.
\end{proposition}

\begin{proof}
	For \( K \) a classical knot \( \dkh ' ( K ) \) decomposes as
	\begin{equation*}
		\dkh ' ( K ) = Kh ' ( K ) \oplus Kh ' ( K ) \lbrace - 1 \rbrace
	\end{equation*}
	so that clearly \( s^{\text{u}}_{\text{max}} = s_{\text{max}} ( K ) \), where \( s_{\text{max}} ( K ) \) denotes the classical quantity. Then
	\begin{equation*}
		\begin{aligned}
			s ( K ) &= s_{\text{max}} ( K ) - 1 \\
			&= s^{\text{u}}_{\text{max}} ( K ) - 1 \\
			&= s^{\text{l}}_{\text{max}} ( K ).
		\end{aligned}
	\end{equation*}
	That \( s_2 ( K ) = 0 \) is observed on \cpageref{Page:classicalacs}.
\end{proof}

The doubled Rasmussen invariant exhibits the same behaviour with respect to mirror image and connect sum as its classical counterpart.

\begin{proposition}
	\label{Prop:mirror}
	Let \( K \) be a virtual knot \( \overline{K} \) denote its mirror image. Then \( \mathbbm{s} ( K ) = - \mathbbm{s} ( \overline{K} ) \).
\end{proposition}

\begin{proof}
	The statement \( s_1 ( K ) = - s_1 ( \overline{K} ) \) follows, as in the classical case, from the existence of the isomorphism of dual complexes
	\begin{equation*}
	r : ( \mathcal{A} \oplus \mathcal{A} \lbrace -1 \rbrace, m', \Delta', \eta' ) \rightarrow ( (\mathcal{A} \oplus \mathcal{A} \lbrace -1 \rbrace)^\ast, {\Delta'}^\ast, {m'}^\ast, {\eta'}^\ast ).
	\end{equation*}
	That \( s_2 ( K ) = - s_2 ( \overline{K} ) \) is seen as follows: let \( D \) be a diagram of \( K \) with \( n_+ \) positive classical crossings and \( n_- \) negative classical crossings. Let \( \mathscr{S} \) be the alternately colourable smoothing of \( D \), so that \( s_2 ( K ) = | \mathscr{S} | \), the height of \( \mathscr{S} \). Further, notice that
	\begin{equation*}
		\begin{aligned}
			| \mathscr{S} | &= n^u_p + n^o_n - n_- \\
			&= n^u_p + n^o_n - ( n^u_n + n^o_n ) \\
			& = n^u_p - n^u_n
		\end{aligned}
	\end{equation*}
	where
	\begin{equation*}
		\begin{aligned}
			n^u_p &= \text{the number of positive crossings resolved into their unoriented smoothing} \\
			n^o_p &= \text{the number of positive crossings resolved into their oriented smoothing} \\
		\end{aligned}
	\end{equation*}
	and likewise \( n^u_n \) and \( n^o_n \) (for a classical knot \( n^u_p  = n^u_n = 0 \), of course). It is quickly observed that
	\begin{equation*}
		\begin{aligned}
			\overline{n}^u_n &= n^u_p \\
			\overline{n}^o_p &= n^o_n \\
		\end{aligned}
	\end{equation*}
	where \( \overline{n}^\ast_\ast \) denote the corresponding quantities for \( \overline{D} \). Then
	\begin{equation*}
		\begin{aligned}
			| \overline{\mathscr{S}} | &= \overline{n}^u_p - \overline{n}^u_n \\
			&= \overline{n}_+ - \overline{n}^o_p - n^u_p \\
			&= n_- - n^o_n - n^u_p \\
			&= n^u_n - n^u_p \\
			&= - | \mathscr{S} |.
		\end{aligned}
	\end{equation*}
\end{proof}

\begin{proposition}
	\label{Prop:additive}
	Let \( K_1 \) and \( K_2 \) be virtual knots and denote by \( K_1 \# K_2 \) any of their connect sums. Then
	\begin{equation*}
		\mathbbm{s} ( K_1 \# K_2 ) = \mathbbm{s} ( K_1 ) + \mathbbm{s} ( K_2 ).
	\end{equation*}
\end{proposition}

\begin{proof}
	It is readily apparent that \( \mathscr{S} = \mathscr{S}_1 \sqcup \mathscr{S}_2 \), where \( \mathscr{S} \) / \( \mathscr{S}_1 \) / \( \mathscr{S}_2 \) is the alternately colourable smoothing of \( K_1 \# K_2 \) / \( K_1 \) / \( K_2 \). Then \( | \mathscr{S} | = | \mathscr{S}_1 | + | \mathscr{S}_2 |\), which proves the claim regarding \( s_2 ( K_1 \# K_2 ) \).
	
	Let \( \lartial : \dkh ' ( K_1 \# K_2 ) \rightarrow \dkh ' ( K_1 \sqcup K_2 ) \) be the map realised by acting \( \Delta' \) on the cycle as dictated by the connect sum. Regarding \( s_1 ( K_1 \# K_2 ) \), the proof follows in identical fashion to the classical proof when one notices that we only require the existence of \( \lartial \) (as opposed to the short exact sequence used in \cite{Rasmussen2010}).
\end{proof}

\subsection{Relationship with the odd writhe}\label{Subsec:rasoddwrithe}
Kauffman defined the odd writhe of a virtual knot in terms of Gauss diagrams \cite{Kauffman2004b}. In this section we show that the doubled Rasmussen invariant contains the odd writhe.

\begin{definition}
	Let \( D \) be a diagram of a virtual knot and \( G ( D ) \) its Gauss diagram. A classical crossing of \( D \), associated to the chord labelled \(c\) in \( G ( D ) \), is known as \emph{odd} if the number of chord endpoints appearing between the two endpoints of \( c \) is odd. Otherwise it is known as \emph{even}. The \emph{odd writhe} of \( D \) is defined
	\begin{equation*}
		J ( D ) = \sum_{\text{odd crossings of}~D} \text{sign of the crossing}.
	\end{equation*}\CloseDef
\end{definition}

\begin{theorem}
	Let \( D \) be a virtual knot diagram of \(K\). The odd writhe is an invariant of \( K \) and we define
	\begin{equation*}
		J ( K ) \coloneqq J ( D ).
	\end{equation*}
\end{theorem}

The odd writhe of a virtual knot \(K\) provides a quick way to calculate \( s_2 ( K ) \).

\begin{proposition}
	\label{Prop:s2odd}
	Let be \( D \) a diagram of a virtual knot \( K \). Then \( s_2 ( K ) = J ( K ) \).
\end{proposition}

\begin{proof}
	We claim that a classical crossing in \( D \) is odd if and only if it is in its unoriented resolution in the alternately colourable smoothing of \( D \).
	
	\noindent(\( \Rightarrow \)): Let \(c\) denote an odd classical crossing of \(D\). Leaving the crossing from either of the outgoing arcs we must return to a specified incoming arc. Between leaving and returning we have passed through an odd number of classical crossings (which are not \(c\)). Thus the incoming arc must be coloured the opposite colour to the outgoing, and \(c\) is resolved into its unoriented resolution in the both of the alternately coloured smoothings of \(D\), as depicted here:
	\begin{center}
		\includegraphics{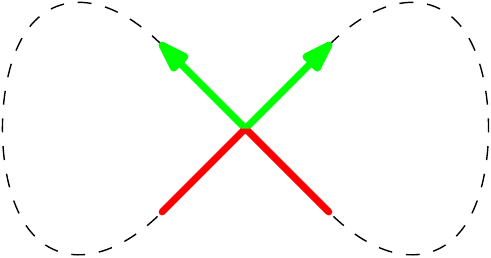}
	\end{center}

	\noindent(\( \Leftarrow \)): Let \(c\) denote a classical crossing of \(D\) which is resolved into its unoriented smoothing in the alternately colourable smoothing of \(D\). The colouring at \(c\) must be as depicted above. Again, leaving \(c\) from either outgoing arc and returning at the specified incoming arc, we see that, as the colours of the arcs are opposite, an odd number of classical crossings must have been passed.
	
	The contributions of odd and even crossings to \( J ( K ) \) and \( s_2 ( K ) \) are summarised in the following table, from which the result follows. The contributions to \( s_2 ( K) \) are clear when one recalls that the height of a smoothing contains the shift \( - n_- \), the total number of negative classical crossings of \(D\).
	\begin{center}
			\begin{tabular}{c | c | c | c | c}
				sign  & parity & reso. & \( J ( K ) \) & \(s_2 ( K )\) \\ \hline
				\(+\) & odd & \(1\) & \( +1 \) & \( +1 \) \\
				\(+\) & even & \( 0 \) & \( 0 \) & \( 0 \) \\
				\(-\) & odd & \(0\) & \( -1 \) & \( -1 \) \\
				\(-\) & even & \(1\) & \( 0 \) & \( 0 \) 
		\end{tabular}
	\end{center}
\end{proof}

\begin{corollary}
	\label{Cor:oddwritheadditive}
	Let \( K_1 \) and \( K_2 \) be virtual knots and \( K_1 \# K_2 \) denote any of their connect sums. Then
	\begin{equation*}
		J ( K_1 \# K_2 ) = J ( K_1 ) + J ( K_2 ).
	\end{equation*}
\end{corollary}

\subsection{Leftmost knots and quick calculations}\label{Subsec:leftmost}
To conclude this section we identity a class of knots for which the calculation of the doubled Rasmussen invariant is trivial, a generalisation of the case of computation of the classical Rasmussen invariant of positive classical knots. The key here, as in the classical case, is that the alternately coloured smoothings of the class of knots in question have no incoming differentials.

\begin{definition}
	Let \( D \) be a virtual knot diagram. We say that \( D \) is \emph{leftmost} if it contains only positive even and negative odd classical crossings. A virtual knot is leftmost if it has a leftmost diagram.\CloseDef
\end{definition}

\begin{proposition}
	\label{Prop:left}
	Let \( D \) be a leftmost diagram of a virtual knot \( K \) with \( n_- \) negative classical crossings. Then \( s_2 ( K ) = - n_- \), the minimal non-trivial homological grading of \( \dkh ' ( K ) \).
\end{proposition}

\begin{proof}
	Let \( D \) be a leftmost diagram of a virtual knot. By \Cref{Prop:s2odd} we have \( s_2 ( K ) = J ( K ) = -n_- \), as a crossing in \( D \) is odd if and only if it is negative.
\end{proof}

\begin{proposition}
	\label{Prop:leftcalc}
	Let \( D \) be a leftmost diagram of a virtual knot \( K \). Then \( s_1 ( K ) = \text{max} \lbrace s ( \sg + \overline{\sg} ), s ( \sg - \overline{\sg} ) \rbrace \), where \( \sg \) is an alternately coloured generator associated to the alternately colourable smoothing of \( D \).
\end{proposition}

\begin{proof}
By \Cref{Prop:left} the alternately colourable smoothing of \( D \) is at the minimal non-trivial height of the cube of resolutions. By construction there is only one smoothing at this height. Further, this smoothing has no incoming differentials. Recalling \Cref{Def:rasinvariant}, we obtain the result.
\end{proof}

\section{Applications}\label{Sec:applications}
We shall now describe some applications of the invariants \( \dkh ( L ) \) and \( \mathbbm{s} ( K )  \). All of the given applications are related to virtual link concordance, to a greater or lesser extent.

\subsection{Cobordism obstructions}\label{Subsec:obstructions}
As mentioned in \Cref{Subsec:cobordisms}, we can use the information contained in the quantum degree of \( \dkh ' ( L ) \) to obtain obstructions to the existence of cobordisms between \( L \) and other links. First we repeat the procedure used to show that the classical Rasmussen invariant yields a bound on the slice genus to obtain a bound on the genus of a certain class of cobordisms from a knot to the unknot, and between two given knots. We then obtain an obstruction to the existence of a concordance between a link and a given knot. Finally, we use doubled Lee homology to show that virtual knots with non-zero odd writhe are not slice.

\subsubsection{Genus bounds}
\label{Subsec:genusbounds}
In this section we use the fact that concordances and targeted cobordisms are assigned non-zero maps to obtain obstructions to the existence of cobordisms of certain genera between pairs of virtual knots. First we obtain a lower bound on the genus of targeted cobordisms between pairs of knots whose \( s_2 \) invariants agree.

\begin{theorem}
	\label{Thm:vgenusbound}
	Let \( K \) be a virtual knot with \( s_2 ( K ) = 0 \) and \( S \) a targeted cobordism from \( K \) to the unknot such that \(0\) is a shared degree of \(S\). Then
	\begin{equation}
	\label{Eq:genusbound}
	\dfrac{| s_1 ( K ) | }{2} \leq g(S).
	\end{equation}
\end{theorem}

\begin{proof}
	Let \( K \) and \( S \) be as in the theorem statement. Then, by \Cref{Thm:nonzero2}, \( \phi_S \) is a non-zero map. As in the classical case, it is easy to see that \( \phi_S \) is filtered of degree \( -2g(S) \). Let \( x \in \dkh ' ( K ) \) realise \( s^{\text{u}}_{\text{max}} ( K ) \) so that
	\begin{equation*}
	1 \geq s ( \phi_S ( x ) ) \geq s^{\text{u}}_{\text{max}} ( K ) -2g(S)
	\end{equation*}
	as \( s^{\text{u}}_{\text{max}} ( \raisebox{-1.75pt}{\includegraphics[scale=0.3]{unknotflat.pdf}} ) = 1 \). This yields
	\begin{equation*}
	\begin{aligned}
	2g(S) + 1 &\geq s^{\text{u}}_{\text{max}} ( K ) \\
	2g(S) &\geq s_1 ( K ).
	\end{aligned}
	\end{equation*}
	Repeating the argument for \( \overline{K} \), and using \Cref{Prop:mirror}, we obtain
	\begin{equation*}
	-2g(S) \leq s_1 ( K )
	\end{equation*}
	which yields the desired result.
\end{proof}

\begin{corollary}
	Let \( K \) be a virtual knot with \( s_2 ( K ) = 0 \) and \( S \) a targeted cobordism from \( K \) to the unknot such that \( 2 g ( S ) \leq | s_1 ( K ) | \). Then there exists a link \( L \) which appears in \( S \) with \( {\dkh '}_0 ( L ) = 0 \).
\end{corollary}

In a very similar manner we able to show the following.

\begin{theorem}
	\label{Thm:vgenusbound2}
	Let \( K_1 \) and \( K_2 \) be a pair of virtual knots with \( s_2 ( K_1 ) = s_2 ( K_2 ) \), and \( S \) be a targeted cobordism between them such that \( s_2 ( K ) \) is a shared homological degree of \(S\). Then
	\begin{equation*}
		\dfrac{| s_1 ( K_1 ) - s_1 ( K_2 ) |}{2} \leq g ( S ).
	\end{equation*}
\end{theorem}

Further, concordances between virtual knots are obstructed by the quantum degree component of the doubled Rasmussen invariant, \( s_1 \) (in \Cref{Subsubsec:oddwritheslice} we show that the homological component is such an obstruction, also).

\begin{theorem}
	\label{Thm:conceobs}
	Let \( K \) and \( K' \) be virtual knots such that \( s_2 ( K ) = s_2 ( K' ) \). If \( s_1 ( K ) \neq s_1 ( K' ) \) then \(K\) and \(K'\) are not concordant.
\end{theorem}

The proof of \Cref{Thm:conceobs} follows almost exactly along the lines of that of \Cref{Thm:vgenusbound}, which itself is very similar to the classical case; all we require is that the map assigned to a concordance is non-zero, which is verified in \Cref{Thm:nonzero1}.

\begin{corollary}
	\label{Cor:sliceobs}
	Let \( K \) be a virtual knot with \( s_2 ( K ) = 0 \). If \( s_1 ( K ) \neq 0 \) then \(K\) is not slice.
\end{corollary}

\subsubsection{Obstructions to concordances between knots and links}
\label{Subsec:combining}

We can extend \Cref{Thm:conceobs} to the case in which one end of the concordance is a link, provided the homologies of the knot and link in question are compatible, and the concordance is connected.

\begin{theorem}
	\label{Thm:merging}
	Let  \( L \) be a virtual link of \( | L | \) components. Further, let \( S \) be a connected concordance between \( L \) and a virtual knot \( K \) such that \( {\dkh '}_{s_2 ( K ) } ( L ) \neq 0 \). Let \( M(L) \) be the maximum non-trivial quantum degree of elements \(x \in \dkh ' ( L ) \) such that \( \phi_S ( x ) \neq 0 \). Then
	\begin{equation*}
	M ( L ) \leq s_1( K ) + | L |.
	\end{equation*}
\end{theorem}

\begin{proof}
	Let \( L \), \( K \), and \( S \) be as in the theorem statement. Then \( \phi_S \) is non-zero by \Cref{Thm:nonzero1}. It is clear that \( \phi_{S} \) is filtered of degree \( - ( | L | - 1 ) \): a minimum of \( | L | - 1 \) \(1\)-handles are needed to take a \( | L | \)-component link to a knot, and any surplus \(1\)-handles must be paired with \(2\)-handles. It is also clear that if \( x \in \dkh ' ( L ) \) is such that \( \phi_{S} ( x ) \neq 0 \) then \( x \in {\dkh '}_{ s_2 ( K ) } ( L ) \). For such an \( x \) we have that \( s ( x ) \geq M ( L ) \) and
	\begin{equation*}
		\begin{aligned}
			M ( L ) - ( | L | - 1 ) \leq s ( x ) - ( | L | - 1 ) \leq s ( \phi_S ( x ) ) \leq s^{\text{u}}_{\text{max}} ( K )
		\end{aligned}
	\end{equation*}
	so that
	\begin{equation*}
		\begin{aligned}
			M ( L ) - | L | + 1 \leq s_1 ( K ) + 1
		\end{aligned}
	\end{equation*}
	as required.
\end{proof}

\begin{corollary}
	\label{Cor:concob}
	Let \( L \) be a virtual link of \( | L | \) components such that \( \dkh ' ( L ) \neq 0 \). Further, let \( K \) a virtual knot such that \( \dkh ' ( L ) \) is trivial in homological degree \( s_2 ( K ) \) or
	\begin{equation*}
		M ( L ) \geq s_1( K ) + | L |.
	\end{equation*}
	Then any concordance from \( L \) to \( K \) is disconnected.
\end{corollary}

A particular consequence of \Cref{Cor:concob} is that, given a virtual link \( L \) for which \( \dkh ' ( L ) \neq 0 \) and \( {\dkh '}_0 ( L ) = 0 \), all concordances from \( L \) to classical knots must be disconnected: no classical knots can be obtained from \(L\) by simply merging its components.

\subsubsection{The odd writhe is an obstruction to sliceness}\label{Subsubsec:oddwritheslice}
The odd writhe of a knot is very easy to calculate. Despite this it can detect non-classicality (and hence non-triviality) and chirality of many virtual knots \cite{Kauffman2004b}. Here we show that it also contains information regarding the concordance class of a virtual knot.

\begin{theorem}
	\label{Thm:owritheslice}
	Let \( K \) be a virtual knot. If \( J ( K ) \neq 0 \) then \( K \) is not slice.
\end{theorem}

\begin{proof}
	We prove the contrapositive. Assume towards a contradiction that \( K \) is a slice virtual knot such that \( J ( K ) \neq 0 \). Then \( s_2 ( K ) \neq 0 \) by \Cref{Prop:s2odd}. Let \( S \) realise a slice disc so that \( \phi_S \) is non-zero by \Cref{Thm:nonzero1}. Recall that \( \phi_S \) preserves homological degree by construction. There must exist \( x \in {\dkh '}_{s_2(K)} ( K ) \) such that \( \phi_S ( x ) \neq 0 \). But then
	\begin{equation*}
		\phi_S ( x ) \neq 0 \in {\dkh '}_{s_2(K)} \left( \raisebox{-1.75pt}{\includegraphics[scale=0.3]{unknotflat.pdf}} \right) = 0
	\end{equation*}
	as \( s_2 ( K ) \neq 0 \), a contradiction.
\end{proof}

The proof of \Cref{Thm:owritheslice} can be used \textit{mutatis mutandis} to show that the set of concordance classes of virtual knots is partitioned by the odd writhe.

\begin{theorem}
	\label{Thm:owritheconc}
	Let \( K_1 \) and \( K_2 \) be virtual knots. If \( J ( K_1 ) \neq J ( K_2 ) \) then \( K_1 \) and \( K_2 \) are not concordant.
\end{theorem}

\begin{corollary}
	\label{Cor:classicalconcordance}
	Let \( K \) be a virtual knot. If \( J ( K ) \neq 0 \) then \( K \) is not concordant to a classical knot. 
\end{corollary}

\subsubsection{Examples}
Consider the classical knot \( T ( 4,3 ) \), as given in \Cref{Fig:819}. By converting a particular subset of its crossings to virtual crossings we are able to produce a virtual knot, \(K\), whose alternately colourable smoothing is its oriented smoothing (\(K\) is also positive, as \( T ( 4,3 ) \) is). Thus \( s_2 (K) = 0 \) and the odd writhe provides no obstruction to sliceness. However, \( K \) is a leftmost knot (as defined in \Cref{Subsec:leftmost}), so that \( s_1 ( K ) = \text{max} \lbrace s ( \sg^{\text{l}} + \overline{\sg}^{\text{l}} ), s ( \sg^{\text{l}} - \overline{\sg}^{\text{l}} ) \rbrace \) by \Cref{Prop:leftcalc}. It can be quickly verified that \( \text{max} \left( s ( \sg^{\text{l}} + \overline{\sg}^{\text{l}} ), s ( \sg^{\text{l}} - \overline{\sg}^{\text{l}} ) \right) = 1 \) so that \( K \) is not slice by \Cref{Cor:sliceobs}.

Further, consider the classical two-component link \( 9^2_{61} \), as depicted in \Cref{Fig:9261}. By an argument identical to that used in the case of leftmost knots we can show that the maxiumum quantum degree of all elements in \( {\dkh '}_0 ( L ) \) is \( 5 \). In the context of \Cref{Thm:merging}, considering connected concordances from \( L \) to the unknot, \( M ( L ) = 5 \) and as \( | L | = 2, s_1 \left( \raisebox{-1.75pt}{\includegraphics[scale=0.3]{unknotflat.pdf}} \right) = 0 \), it follows that there does not exist a connected concordance from \( L \) to the unknot.

The method used in both the above examples can be applied to many positive oriented classical link diagrams in order to produce virtual link diagrams for which the quantum degree information (at particular homological degrees) is easy to compute.

\begin{figure}
	\includegraphics[scale=0.65]{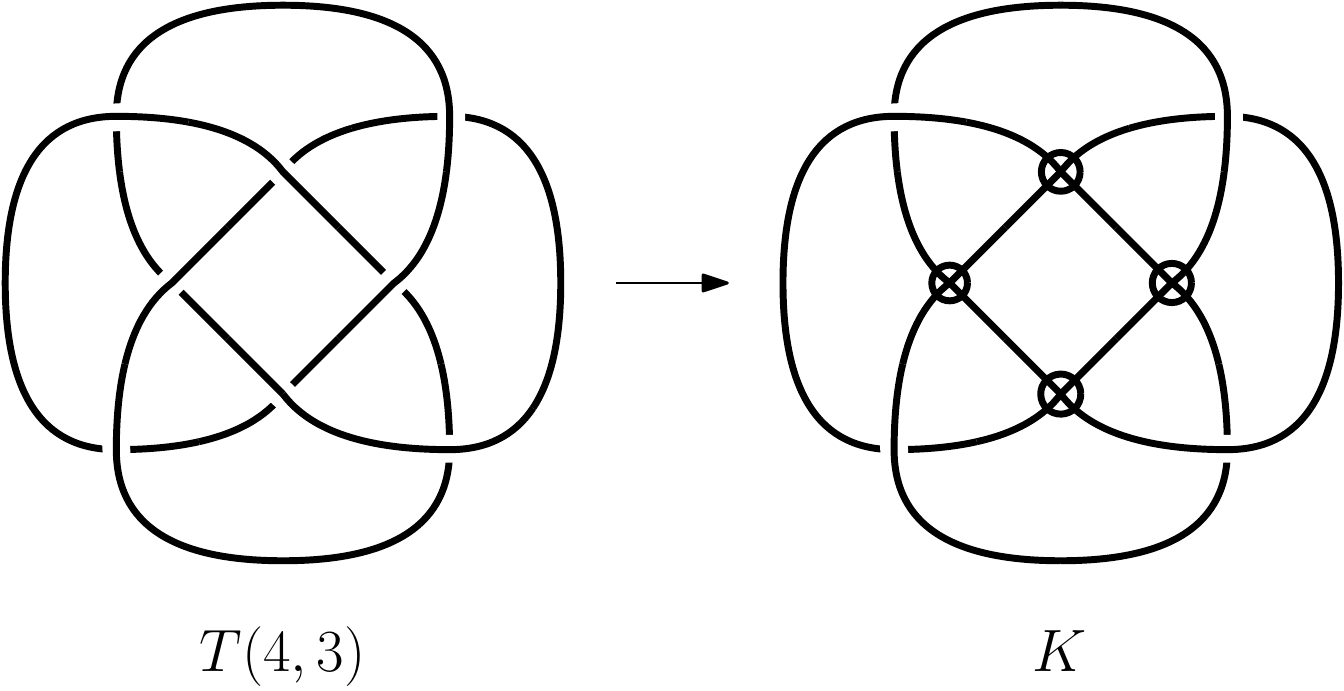}
	\caption{The classical torus knot \( T ( 4,3 ) \), on the left, and a virtual knot \(K\) formed by converting a subset of its crossings to virtual crossings, on the right.}
	\label{Fig:819}
\end{figure}

\begin{figure}
	\includegraphics[scale=1]{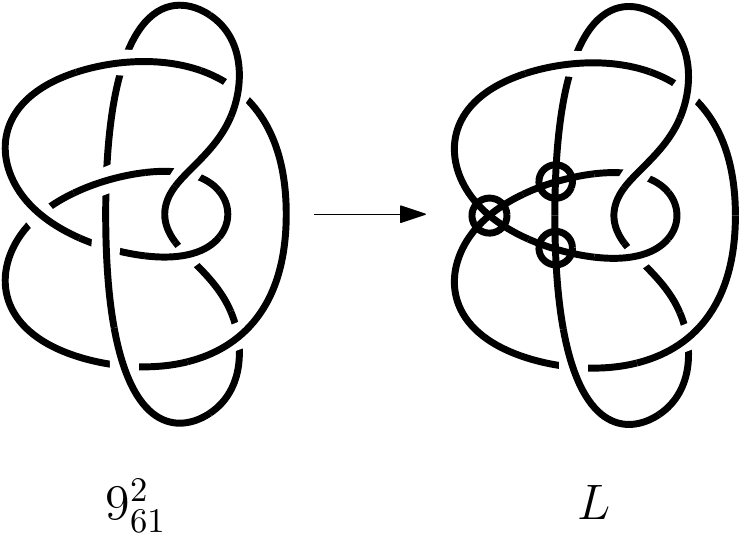}
	\caption{The classical link \( 9^2_{61} \), on the left, and a virtual link \(L\) formed by converting a subset of its crossings to virtual crossings, on the right.}
	\label{Fig:9261}
\end{figure}

\subsection{Connect sums of trivial diagrams}\label{Subsec:trivialdiagrams}
First let us recall the definition of the connect sum of virtual knot diagrams (see \cite{Manturov2008}).

\begin{definition}
	\label{Def:csum}
	Let \( D_1 \) and \( D_2 \) be oriented virtual knot diagrams. If \( D_1 \sqcup D_2 \hookrightarrow \mathbb{R}^2 \) is such that there exists a disc \(B \hookrightarrow \mathbb{R}^2 \) with \( B \cap D_1 = I \) and \( B \cap D_2 = \overline{I} \) (where \( I \) denotes an oriented unit interval and \( \overline{I} \) an interval with reverse orientation) then we denote by \( D_1 \#_B D_2 \) the diagram produced by \(1\)-handle addition with attaching sphere \( S^0 \times D^1 = I \sqcup \overline{I} \).\CloseDef
\end{definition}

The connect sum operation is well defined on classical knots. As mentioned in \Cref{Subsec:summary}, this is not the case for virtual knots. Given a pair of virtual knots \(K_1\) and \(K_2\) with diagrams \( D_1 \) and \(D_2\), respectively, the result of the connect sum operation \( D_1 \#_B D_2 \) depends on the diagrams used and the choice of disc \(B\). By an abuse of notation we use \(K_1 \# K_2 \) to refer to any of the knots produced by a connect sum operation on \( K_1 \) and \( K_2 \).

\begin{figure}
	\includegraphics[scale=0.5]{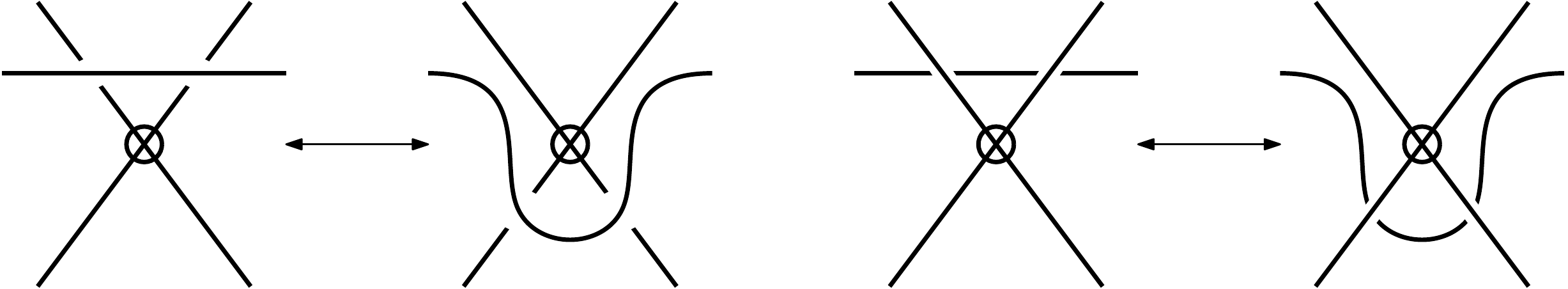}
	\caption{The forbidden moves.}
	\label{Fig:forbidden}
\end{figure}

There are two ways in which to interpret the ill-defined nature of the connect sum operation on virtual knots. The first is in a diagrammatic manner: no longer can one area of the diagram be freely moved over all others, due to presence of the forbidden moves. These are moves on diagrams, depicted in \Cref{Fig:forbidden}, which do not follow from the virtual Reiedemeister moves (in fact, they can be used to unknot any virtual knot \cite{Nelson2001}). Classically, Reidemeister moves commute, in a certain sense, with handle addition: for example, let \( D_1 \) and \( D_2 \) be classical unknot diagrams. Then \( D_2 \) can be treated as a small neighbourhood of \( D_1 \# D_2 \) and slid under (or over) the rest of the diagram. Thus the sequence of Reidemeister moves which takes \( D_1 \) to the crossingless unknot diagram can be replicated on \( D_1 \# D_2 \), taking it to \( D_2 \), which is itself an unknot diagram. Nontrivial diagrams are treated similarly. Virtually, however, this cannot be replicated as areas of a diagram cannot always be moved across others.

The second, deeper, interpretation is as a consequence of the higher-dimensional topological information constituting a virtual knot: as mentioned above, a virtual knot is an equivalence class of embeddings of \( S^1 \) into thickened closed orientable genus \(g\) surfaces, up to isotopy of the thickened surface and handle stabilisations of the (unthickened) surface \cite{Kauffman1998}. Not only does a virtual knot depend on how the copy of \(S^1\) is knotted about itself but also on how it is `knotted' about the topology of the thickened genus \(g\) surface into which it is embedded.

In this light we see that the connect sum operation is not only a \(1\)-dimensional \(1\)-handle addition between the copies of \(S^1\), but that it also induces a \(3\)-dimensional \(1\)-handle addition on the thickened genus \(g\) surfaces\footnote{it is possible for the connect sum operation not to induce a \(3\)-dimensional handle addition but a slightly more complicated operation. We refer the reader to \cite[page 41, Fig. \(2.7\)]{Manturov2013}.}. This contrasts with the classical case in which both copies of \(S^1\) can be contained in one copy of \(S^3\) and only a \(1\)-dimensional \(1\)-handle need be added. Different choices of the disc \(B\) (as in \Cref{Def:csum}) correspond to different choices of \(3\)-dimensional handles. The author suspects that the dependence of the connect sum operation on both the diagrams involved and the choice of \( B \) is inherited ultimately from the non-triviality of the fundamental group of a genus \(g\) surface.

A novel manifestation of this ill-definedness is that there exist non-trivial virtual knots which are connect sums of a pair of trivial virtual knots. The first example of this is given by Kishino's knot \cite{Kishino2004} as depicted in \Cref{Fig:kknot}. Doubled Khovanov homology yields a condition on a virtual knot being a connect sum of two trivial diagrams.

\begin{figure}
	\includegraphics[scale=0.75]{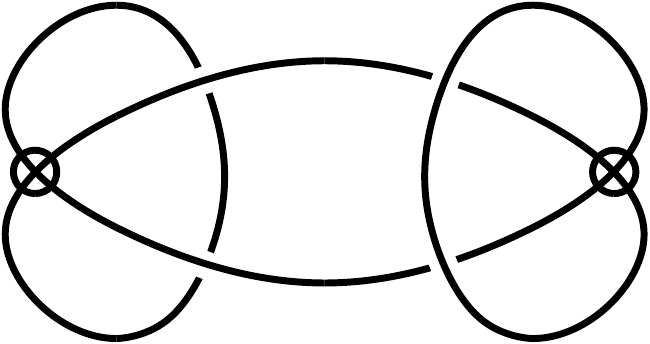}
	\caption{Kishino's knot}
	\label{Fig:kknot}
\end{figure}

\begin{theorem}
	\label{Thm:unknotcondition}
	Let \( K \) be a virtual knot which is a connect sum of two trivial knots. Then \( \dkh ( K ) = \dkh \left( \raisebox{-1.75pt}{\includegraphics[scale=0.3]{unknotflat.pdf}} \right)\). 
\end{theorem}

In order prove to \Cref{Thm:unknotcondition} we shall define a reduction of doubled Khovanov homology, in direct analogy to the classical case \cite{Khovanov1999,Shumakovitch2011a}.

\begin{definition}[Reduced doubled Khovanov homology]
	Let \( L \) be an oriented virtual link diagram with a marked point on each component (away from the crossings of \(L\)). Distribute these marked points across the cube of smoothings so that each smoothing of \( L \) contains \( | L | \) marked points. Define \( \mathcal{C} ( L ) \) to be the chain subcomplex of \( \cdkh ( L ) \) spanned by those states in which all the marked cycles are decorated with either \( \vum \) or \( \vlm \) (all the marked cycles are decorated with the same algebra element). That \( \mathcal{C} \) is a subcomplex is evident from \Cref{Eq:etamap,Eq:diffcomp} (it is also graded).

	Let \( \mathcal{H} ( L ) \) denote the homology of \( \mathcal{C} ( L ) \). We refer to \( \mathcal{H} ( L ) \) as the \emph{reduced doubled Khovanov homology} of \( L \). \CloseDef
\end{definition}

The proof of invariance of \( \mathcal{H} ( L ) \) under virtual Reidemeister moves follows as in the classical case. Invariance under the choice of basepoints follows similarly.

\begin{lemma}
	\label{Lem:reducedrelation}
	Let \( L \) be a virtual link diagram. Then \( \cdkh ( L ) / \mathcal{C} ( L ) \cong \mathcal{C} ( L ) \lbrace 2 \rbrace \).
\end{lemma}

\begin{proof}
We prove the statement for a virtual knot diagram \( K \) (link diagrams follow essentially identically). Let \( \cdkh ( K ) / \mathcal{C} ( K ) = \mathcal{C}' ( K ) \). The isomorphism \( g : \mathcal{C}' ( K ) \rightarrow \mathcal{C} ( K ) \) is straightforward to define. Given a representative, \( x \), of an element of \( \mathcal{C}' ( K ) \) the marked cycle must be decorated with either \( \vup \) or \( \vlp \) i.e.\ we must have \( x = x_1 \otimes x_2 \otimes \ldots \otimes \vulp \otimes \ldots x_n \). Define
\begin{equation*}
	\begin{aligned}
		g ( x_1 \otimes x_2 \otimes \ldots \otimes \vulp \otimes \ldots x_n ) &= x_1 \otimes x_2 \otimes \ldots \otimes \vulm \otimes \ldots x_n \\
		g^{-1} ( x_1 \otimes x_2 \otimes \ldots \otimes \vulm \otimes \ldots x_n ) &= _1 \otimes x_2 \otimes \ldots \otimes \vulp \otimes \ldots x_n.
	\end{aligned}
\end{equation*}

That \( g \) is well defined is clear and that it is a chain map is apparent when one considers the schematic given in \Cref{Fig:schematic}: the only issue that could arise is due to the factor of \( 2 \) in the \( \eta \) map, the position of which ensures that it does not cause any trouble. That the degree of \( g \) is \( -2 \) is obvious.
\begin{figure}
	\includegraphics[scale=0.8]{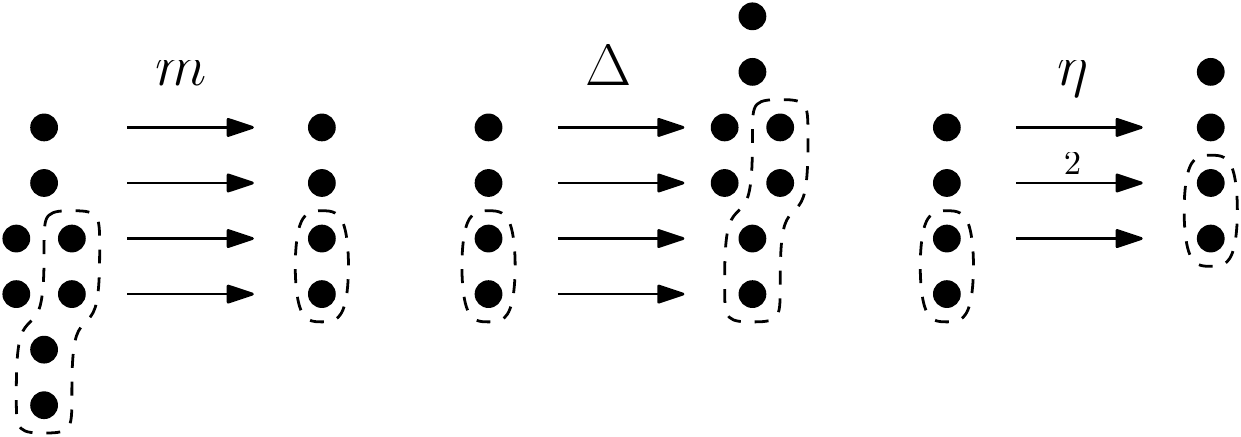}
	\caption{A schematic for the interaction between the map \( g \) and the differential. The enclosed dots depict generators of \( \mathcal{C} ( K ) \); \( g \) sends a dot outside an enclosure to the corresponding dot inside.}
	\label{Fig:schematic}
\end{figure}
\end{proof}

\begin{proof}[Proof of \Cref{Thm:unknotcondition} ]Let \( K \) be as in the proposition. By an abuse of notation let \( K = D_1 \# D_2 \) be the diagram which is the result of a connect sum between \( D_1 \) and \( D_2 \), both of which are unknot diagrams. We are free to pick marked points on the diagrams \( K \) and \( D_1 \sqcup D_2 \) so that the situation is as in \Cref{Fig:reduceddiagram}, from which we observe that there is a chain complex isomorphism from \( f : \mathcal{C} ( K ) \rightarrow \mathcal{C} ( D_1 \sqcup D_2 ) \). The isomorphism is defined as follows
\begin{equation*}
	\begin{aligned}
		f( x_1 \otimes x_2 \otimes \ldots \otimes \vulm \otimes \ldots \otimes x_n ) &=  x_1 \otimes x_2 \otimes \ldots \otimes \vulm \otimes \vulm \otimes \ldots \otimes x_n \\
		f^{-1}( x_1 \otimes x_2 \otimes \ldots \otimes \vulm \otimes \vulm \otimes \ldots \otimes x_n ) &= x_1 \otimes x_2 \otimes \ldots \otimes \vulm \otimes \ldots \otimes x_n
	\end{aligned}
\end{equation*}
where \( \vulm \) and \( \vulm \otimes \vulm \) decorate the marked cycles. That \( f \) is a chain map follows from the observation that if \( \sg^{\text{u/l}} \) is a state of \( \mathcal{C} ( K ) \) then \( f ( \sg^{\text{u/l}} ) \) has the same incoming and outgoing differentials. It is clear that \( f \) is graded of degree \( -1 \).

We have established the isomorphism \( \mathcal{C} ( K ) \cong \mathcal{C} ( D_1 \sqcup D_2 ) \lbrace 1 \rbrace \); further, there is a chain homotopy equivalence between \( \mathcal{C} ( D_1 \sqcup D_2 ) \) and \( \mathcal{C} \left( \raisebox{-2.5pt}{\includegraphics[scale=0.3]{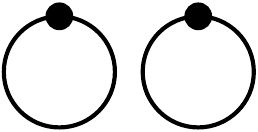}} \right) \) as \( D_1 \) and \( D_2\) are unknot diagrams. It is easy to see that \( \mathcal{C} \left( \raisebox{-2.5pt}{\includegraphics[scale=0.3]{unlink.pdf}} \right) = \mathcal{C} \left( \raisebox{-2.5pt}{\includegraphics[scale=0.3]{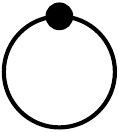}} \right) \lbrace -1 \rbrace \) so that
\begin{equation*}
	\label{Eq:ccr}
	\mathcal{C} ( K ) \cong \mathcal{C} ( D_1 \sqcup D_2 ) \lbrace 1 \rbrace \simeq \mathcal{C} \left( \raisebox{-2.5pt}{\includegraphics[scale=0.3]{unlink.pdf}} \right) \lbrace 1 \rbrace \simeq \left( \mathcal{C} \left( \raisebox{-2.5pt}{\includegraphics[scale=0.3]{unknotmarked.pdf}} \right) \lbrace -1 \rbrace \right) \lbrace 1 \rbrace = \mathcal{C} \left( \raisebox{-2.5pt}{\includegraphics[scale=0.3]{unknotmarked.pdf}} \right)
\end{equation*}
and
\begin{equation}
\label{Eq:hr}
\mathcal{H} ( K ) = \mathcal{H} \left( \raisebox{-2.5pt}{\includegraphics[scale=0.3]{unknotmarked.pdf}} \right).
\end{equation}
In addition, there is an exact triangle
\begin{equation}
\label{Eq:exacttri}
\begin{tikzpicture}[baseline=(current  bounding  box.center)]
\matrix (m) [matrix of math nodes,row sep=30pt,column sep=2.5pt]
{
	\mathcal{H} ( K ) && \dkh( K ) \\
	&\mathcal{H} ( K ) \lbrace 2 \rbrace &\\};
\path[->]
(m-1-1) edge (m-1-3)
(m-2-2) edge (m-1-1)
(m-1-3) edge (m-2-2)  
;
\end{tikzpicture}
\end{equation}
which is arrived at via the short exact sequence
\begin{equation*}
	\begin{tikzpicture}[roundnode/.style={}]
		\node[roundnode] (s0)at (-4,0)  {\( 0 \)
			};
		
		\node[roundnode] (s1)at (-2.5,0)  {\( \mathcal{C} ( K ) \)
			};
		
		\node[roundnode] (s2)at (0.25,0)  {\( \cdkh ( K ) \)
			};
		
		\node[roundnode] (s3)at (3.5,0)  {\( \cdkh ( K ) / \mathcal{C} ( K ) \)
			};
	
		\node[roundnode] (s4)at (6,0)  {\( 0, \)
			};
		
		\draw[->] (s0)--(s1) ;
		
		\draw[->] (s1)--(s2) ;
		
		\draw[->] (s2)--(s3) ;
		
		\draw[->] (s3)--(s4) ;
	\end{tikzpicture}
\end{equation*}
\Cref{Lem:reducedrelation} and the observation that \Cref{Eq:hr} implies that \( \mathcal{H} ( K ) \) is supported in homological degree \( 0 \). Also by \Cref{Eq:hr} we obtain \( \text{rank} ( \mathcal{H} ( K ) )  = 2 \) so that the triangle splits and
\begin{equation*}
	\dkh ( K ) = \mathcal{H} ( K ) \oplus \mathcal{H} ( K ) \lbrace 2 \rbrace = \mathcal{H} \left( \raisebox{-2.5pt}{\includegraphics[scale=0.3]{unknotmarked.pdf}} \right) \oplus \mathcal{H} \left( \raisebox{-2.5pt}{\includegraphics[scale=0.3]{unknotmarked.pdf}} \right) \lbrace 2 \rbrace = \dkh \left( \raisebox{-1.75pt}{\includegraphics[scale=0.3]{unknotflat.pdf}} \right).
\end{equation*}
\begin{figure}
	\includegraphics[scale=1]{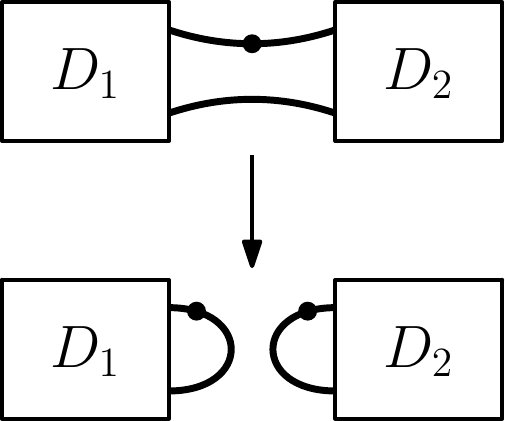}
	\caption{Marked diagrams of \( K \) (above) and \( D_1 \sqcup D_2 \) (below).}
	\label{Fig:reduceddiagram}
\end{figure}
\end{proof}

\begin{proposition}
	\label{Prop:csumred}
	Let \( K \) and \( K' \) be virtual knots which are connect sums of the same pair of initial virtual knots \( J \) and \( J' \): that is, there exist diagrams \( D_1 \) and \( D_2 \) of \( J \) and \( D_3 \) and \( D_4 \) of \( J' \) such that \( K = D_1 \# D_3 \) and \( K' = D_2 \# D_4 \). Then \( \mathcal{C} ( K ) \simeq \mathcal{C} ( K' ) \).
\end{proposition}

\begin{proof}
	We have \( \mathcal{C} ( K ) \cong \mathcal{C} ( D_1 \sqcup D_3 ) \simeq \mathcal{C} ( D_2 \sqcup D_4 ) \cong \mathcal{C} ( K' ) \), as \( D_1 \sqcup D_3 \) and \( D_2 \sqcup D_4 \) are both diagrams of \( J \sqcup J' \) and the isomorphisms are essentially identical to that given in the proof of \Cref{Thm:unknotcondition}.
\end{proof}

\begin{remark}
	Of course, there is still a pair of short exact sequences
	\begin{equation*}
		\begin{tikzpicture}[roundnode/.style={}]
		\node[roundnode] (s0)at (-4,0)  {\( 0 \)
		};
		
		\node[roundnode] (s1)at (-2.5,0)  {\( \mathcal{C} ( K ) \)
		};
		
		\node[roundnode] (s2)at (0.35,0)  {\( \cdkh ( K ) \)
		};
		
		\node[roundnode] (s3)at (3.5,0)  {\( \mathcal{C} ( K ) \lbrace 2 \rbrace \)
		};
		
		\node[roundnode] (s4)at (5.5,0)  {\( 0, \)
		};
	
		\node[roundnode] (t0)at (-4,-1.5)  {\( 0 \)
		};
		
		\node[roundnode] (t1)at (-2.5,-1.5)  {\( \mathcal{C} ( K' ) \)
		};
		
		\node[roundnode] (t2)at (0.35,-1.5)  {\( \cdkh ( K' ) \)
		};
		
		\node[roundnode] (t3)at (3.5,-1.5)  {\( \mathcal{C} ( K' )\lbrace 2 \rbrace \)
		};
		
		\node[roundnode] (t4)at (5.55,-1.5)  {\( 0, \)
		};
		
		\draw[->] (s0)--(s1) ;
		
		\draw[->] (s1)--(s2) ;
		
		\draw[->] (s2)--(s3) ;
		
		\draw[->] (s3)--(s4) ;
		
		\draw[->] (t0)--(t1) ;
		
		\draw[->] (t1)--(t2) ;
		
		\draw[->] (t2)--(t3) ;
		
		\draw[->] (t3)--(t4) ;
		
		\draw[->] (s1)--(t1) node[above,pos=0.5,rotate=-90]{\( \simeq \)			
		};
				
		\draw[->] (s3)--(t3) node[above,pos=0.5,rotate=-90]{\( \simeq \)			
		};
		
		\end{tikzpicture}
	\end{equation*}
but the associated long exact sequences no longer split. Indeed, it is not true in general that
\begin{equation*}
	\dkh ( K ) = \mathcal{H} ( K ) \oplus \mathcal{H} ( K ) \lbrace 2 \rbrace,
\end{equation*}
the aforementioned virtual knot \( 2.1 \) provides a counterexample.
\end{remark}

\bibliographystyle{plain}
\bibliography{library}

\end{document}